\documentclass[11pt,a4paper]{article}
\usepackage{a4wide,amsfonts,amsmath,latexsym,amsthm,amssymb,euscript,eufrak,graphicx,units,mathrsfs,setspace,stmaryrd}

\usepackage{float}
\newfloat{figure}{H}{lof}
\floatname{figure}{\figurename}

\usepackage[french,english]{babel}
\usepackage[T1]{fontenc}
\usepackage[utf8]{inputenc}

\DeclareMathAlphabet{\eufrak}{U}{}{}{}  
\SetMathAlphabet\eufrak{normal}{U}{euf}{m}{n}
\SetMathAlphabet\eufrak{bold}{U}{euf}{b}{n}

\newtheorem{prop}{Proposition}[section]

\newtheorem{theorem}[prop]{Theorem}
\newtheorem{lemma}[prop]{Lemma}
\newtheorem{corollary}[prop]{Corollary}

\newtheorem{assumption}[prop]{Assumption}

\theoremstyle{definition}

\newtheorem{remark}[prop]{Remark}
\newtheorem{remarks}[prop]{Remarks}
\newtheorem{definition}[prop]{Definition}
\newtheorem{discussion}[prop]{Discussion}
\newtheorem{notation}[prop]{Notation}
\newtheorem{convention}[prop]{Convention}
\newtheorem{definitionprop}[prop]{Definition-Proposition}

\numberwithin{equation}{section}

\def\E{\mathbb{E}}
\def\P{\mathbb{P}}
\def\real{\mathbb{R}}

\def\F{\mathcal{F}}
\def\1{\textbf{1}}
\def\ind#1{\textbf{1}_{\{#1\}}}

\newcommand{\be}{\begin{equation}}
\newcommand{\ee}{\end{equation}}
\newcommand{\bde}{\begin{displaymath}}
\newcommand{\ede}{\end{displaymath}}
\newcommand{\beq}{\begin{eqnarray*}}
\newcommand{\eeq}{\end{eqnarray*}}
\newcommand{\beqa}{\begin{eqnarray}}
\newcommand{\eeqa}{\end{eqnarray}}
\newcommand{\bel }{\left\{\begin{array}{ll}}
\newcommand{\eel}{\cr \end{array} \right.}

\newcommand{\bex}{\begin{ex} \rm }
\newcommand{\eex}{\end{ex}}
{

\def\E{\mathbb E}
\def\F{{\cal F}}

\def\P{\mathbb P}

\def\cal#1{\mathcal{#1}}

\DeclareSymbolFontAlphabet{\mathrsfs}{rsfs}

\author{Caroline Hillairet\footnote{ENSAE  Paris, CREST UMR 9194,
5  avenue Henry Le Chatelier
91120 Palaiseau, France.  Email: \texttt{caroline.hillairet@ensae.fr}} \and Anthony R\'eveillac\footnote{INSA de Toulouse, IMT UMR CNRS 5219, Universit\'e de Toulouse, 135 avenue de Rangueil 31077 Toulouse Cedex 4 France. \; Email: \texttt{anthony.reveillac@insa-toulouse.fr}} }

\title{On the chaotic expansion for counting processes  \footnote{This research is supported by a grant of the French National Research Agency (ANR), “Investissements d’Avenir” (LabEx Ecodec/ANR-11-LABX-0047) and the Joint Research Initiative "Cyber Risk Insurance: actuarial modeling" with the partnership of AXA Research Fund. }}

\begin{document}

\maketitle

\allowdisplaybreaks

\begin{abstract}
\noindent
We introduce and study an alternative form of the chaotic expansion for counting processes using the Poisson imbedding representation; we name this alternative form \textit{pseudo-chaotic expansion}. As an application, we prove that the coefficients of this pseudo-chaotic expansion for any linear Hawkes process are obtained in closed form, whereas those of the usual chaotic expansion cannot be derived explicitly. Finally, we study further the structure of linear Hawkes processes by constructing an example of a process in a pseudo-chaotic form that satisfies the stochastic self-exciting intensity equation which determines a Hawkes process (in particular its  expectation equals the one of a Hawkes process) but which fails to be a counting process.
\end{abstract}


\section{Introduction}

Counting processes constitute a mathematical framework for modeling specific random events within a time series. The nature of the application and the structural features of a given time series may call for sophisticated models whose complexity goes beyond the standard Poisson process. For instance, the modeling of neuron's spikes in Neurosciences (see \textit{e.g.} \cite{Bremaud_Massoulie,Delattre_2016}) or the frequency of claims that may result in a cyber insurance contract (see \cite{Hillairet_Reveillac_Rosenbaum} for a short review on the literature) require counting processes with stochastic intensity  and whose frequency of the events {depends on} the past values of the system. Hawkes processes initially introduced in \cite{Hawkes} has become the paradigm of such processes. The so-called linear Hawkes process is a counting process $H$ (on a filtered probability space) with intensity process $\lambda$ satisfying :
\begin{equation}
\label{eq:introlambdaHawkes}
\lambda_t = \mu + \int_{(0,t)} \Phi(t-s) dH_s, \quad t\geq 0,
\end{equation}
where the constant $\mu>0$ is the baseline intensity and $\Phi:\real_+ \to \real_+$ is modeling the self-exciting feature of the process. Naturally conditions on $\Phi$ are required for a well-posed formulation. Well-posedness here is accurate as, from this formulation it appears that the pair $(H,\lambda)$ solves a two-dimensional SDE driven by a Poisson measure as we will make precise below. The so-called Poisson imbedding provides a way to formulate this equation. Let $(\Omega,\mathcal F,\P)$ be a probability space and $N$ a random Poisson measure on $\real_+^2$ with intensity $d\lambda(t,\theta):=dt d\theta$ the Lebesgue measure on $\real_+^2$. If $\lambda$ denotes a non-negative predictable with respect to the natural history of $N$ (whose definition will be recalled in Section \ref{section:preliminaries}) then the process $H$ defined as 
\begin{equation}
\label{eq:introimbedding}
H_t = \int_{(0,t]\times \real_+} \ind{\theta \leq \lambda_s} N(ds,d\theta), \quad t\geq 0,
\end{equation}
is a counting process with intensity $\lambda$, that is $H-\int_0^\cdot \lambda_s ds$ is a martingale. The Poisson imbedding refers to Representation (\ref{eq:introimbedding}) for counting processes. As mentioned, in case of a linear Hawkes process for instance, the Poisson imbedding representation (\ref{eq:introimbedding}) captures the equation feature of this process. Indeed combining (\ref{eq:introlambdaHawkes}) and (\ref{eq:introimbedding}) implies that the linear Hawkes processes can be seen as a system of weakly coupled SDEs with respect to $N$ as :
$$
\left\lbrace
\begin{array}{l}
H_t = \int_{(0,t]\times \real_+} \ind{\theta \leq \lambda_s} N(ds,d\theta) \\ \hspace{20em} t\geq 0, \\
\lambda_t = \mu + \int_{(0,t)} \Phi(t-s) \ind{\theta \leq \lambda_s} N(ds,d\theta).
\end{array}
\right.
$$
Under mild condition on $\Phi$ the second equation (and so the system) can be proved to be well-posed (see \textit{e.g.} \cite{Bremaud_Massoulie,Costa_etal,Hillairet_Reveillac_Rosenbaum}).\\\\
\noindent 
This representation also opens the way to a new line of research. Indeed, with this representation at hand, a counting process can then be seen as a functional of the two-dimensional Poisson measure $N$ for which stochastic analysis such as the Malliavin calculus is available. Recently, by combining the Malliavin calculus with Stein's method according to the Nourdin-Peccati methodology, quantitative limit theorems for Hawkes functionals have been derived in \cite{torrisi,HHKR,Khabou_Privault_Reveillac}. The specific Malliavin calculus developed in \cite{HHKR} for linear Hawkes processes follows from a Mecke formula provided in \cite{Hillairet_Reveillac_Rosenbaum} with application to the risk analysis of a class of cyber insurance contracts. Another main ingredient (not exploited so far in this context up to our knowledge) when dealing with Gaussian and Poisson functionals is given by the so-called chaotic expansion (also called Wiener-It\^o expansion). Let $t \geq 0$ and $H_t$ as in (\ref{eq:introimbedding}). The chaotic expansion of $H_t$ with respect to $N$ writes down as : 
\begin{equation}
\label{eq:introchaos}
H_t = \E[H_t] + \sum_{j=1}^{+\infty} \frac{1}{j!}  \int_{[0,t]\times \real_+} \cdots \int_{[0,t]\times \real_+}  f_j(x_1,\ldots,x_{j}) (N(dx_{1})-dx_1) \cdots (N(dx_{j})-dx_j),
\end{equation}
where $x_i \in \real_+^2$ and $f_j$ is a symmetric function on $(\real_+^2)^j$ defined in terms of the $j$th Malliavin derivative of $H_t$. This expression requires some details on its definition that will be given in Section \ref{section:preliminaries} below; but roughly speaking it allows one to expand the random variable into iterated integrals with respect to the compensated Poisson measure $N(dx_{j})-dx_j := N(dt_{j},d\theta_j)-dt_{j} d\theta_j$ (with $x_j:=(t_j,\theta_j)$).  Such decomposition is proved to be useful (for example in the context of Brownian SPDEs) provided that the coefficients $f_j$ can be computed or can be characterized by an equation. In case of a linear Hawkes process we will show in Section \ref{section:Hawkes} that the coefficients can be computed but are far from being explicit.\\\\\noindent
In this paper we prove as Theorem \ref{th:pseudochaoticcounting} that counting processes satisfy an alternative representation of the chaotic expansion that we name pseudo-chaotic representation. This pseudo chaotic expansion takes the form of : 
\begin{equation}
\label{eq:intropseudochaos}
H_t = \sum_{j=1}^{+\infty} \frac{1}{j!}  \int_{[0,t]\times \real_+} \cdots \int_{[0,t]\times \real_+}  c_j(x_1,\ldots,x_{j}) N(dx_{1}) \cdots N(dx_{j}),
\end{equation}
involving iterated integrals of the counting measure $N$ only (and not its compensated version). Whereas any square integrable random variable $F$ admits a chaotic expansion of the form (\ref{eq:introchaos}) we characterize in Theorem \ref{th:characpseudo} those random variables for which a pseudo-chaotic expansion of the form (\ref{eq:intropseudochaos}) is valid. As mentioned any variable $F=H_t$ with $H$ a counting process belongs to this set. \\\\
\noindent
In case where $H$ is a linear Hawkes process, in contradistinction to the coefficients $f_j$ in the classical chaotic expansion (\ref{eq:introchaos}) of $H_t$ at some time $t$ which can not be computed explicitly, coefficients $c_j$ of the pseudo-chaotic expansion (\ref{eq:intropseudochaos}) are explicit and given in Theorem \ref{th:explicitHawkes} (see Discussion \ref{discussion:avantageforHawkes}). This provides then a closed-form expression to linear Hawkes processes. Finally, we study further in Section \ref{section:almostHawkes} the structure of linear Hawkes processes by constructing an example of a process in a pseudo-chaotic form that satisfies the stochastic intensity equation (\ref{eq:introlambdaHawkes}) but which fails to be a counting process (see Theorem \ref{th:almostHawkes} and Discussion \ref{discussion:finale}).\\\\
\noindent
The paper is organized as follows. Notations and the description of the Poisson imbedding together with elements of Malliavin's calculus are presented in Section \ref{section:preliminaries}. The notion of pseudo-chaotic expansion is presented in Section \ref{section:pseudochaotic}. The application to linear Hawkes processes and their explicit representation is given in Section \ref{section:Hawkes}. Finally, Section \ref{section:almostHawkes} is dedicated to the construction of an example of a process in a pseudo-chaotic form that satisfies the stochastic intensity equation but which fails to be a counting process.  
 
\section{Preliminaries and notations}
\label{section:preliminaries}

\subsection{General conventions and notations}

We set $\mathbb{N}^*:=\mathbb{N} \setminus \{0\}$ the set of positive integers. We make use of the convention :

\begin{convention}
\label{convention:sums}
For $a, b \in \mathbb Z$ with $a > b$, and for any map $\rho : \mathbb Z \to \real$,
$$ \prod_{i=a}^b \rho(i) :=1; \quad \sum_{i=a}^b \rho(i) :=0.$$
\end{convention}

We set 
\begin{equation}
\label{eq:X}
\mathbb X:= \real_+\times \real_+ = \{x=(t,\theta), \; t \in \real_+, \; x\in \real_+\};
\end{equation}
Throughout this paper we will make use of the notation $(t,\theta)$ to refer to the first and second coordinate of an element in $\mathbb X$.

\begin{notation}
\label{notation:ordered}
Let $k\in \mathbb N^*$ and $(x_1,\ldots,x_k)=((t_1,\theta_1),\ldots,(t_k,\theta_k))$ in $\mathbb X^k$. We set $(x_{(1)},\ldots,x_{(k)})$ the ordered in the $t$-component of $(x_1,\ldots,x_k)$ with \;
$0 \leq t_{(1)} \leq \cdots \leq t_{(k)} ,$\;
and write $x_{(i)}:=(t_{(i)},\theta_{(i)})$. 
\end{notation}
\noindent We simply write $dx:=dt \, d\theta$ for the Lebesgue measure on $\mathbb X$. We also set $\mathcal B(\mathbb X)$ the set of Borelian of $\mathbb X$.

\subsection{Poisson imbedding and elements of Malliavin calculus}
Our approach lies on the so-called Poisson imbedding representation allowing one to represent a counting process with respect to a baseline random Poisson measure on $\mathbb X$. Most of the elements presented in this section are taken from \cite{Privault_2009,Last2016}.\\
\noindent 
We define $\Omega$ the space of configurations on $\mathbb X$ as 
$$ \Omega:=\left\{\omega=\sum_{i=1}^{n} \delta_{x_i}, \; x_i:=(t_{i},\theta_i) \in \mathbb X,\; i=1,\ldots,n,\; 0=t_0 < t_1 < \cdots < t_n, \; \theta_i \in \mathbb{R}_+, \; n\in \mathbb{N}\cup\{+\infty\} \right\}.$$
Each path of a counting process is represented as an element $\omega$ in $\Omega$ which is a $\mathbb N$-valued $\sigma$-finite measure on $\mathbb X =\mathbb{R}_+^2$. Let $\mathcal F$ be the $\sigma$-field associated to the vague topology on $\Omega$. Let $\P$ the Poisson measure on $ \Omega$ under which the canonical process $N$ on $\Omega$ is a Poisson process with intensity one that is : 
$$ (N(\omega))([0,t]\times[0,b])(\omega):=\omega([0,t]\times[0,b]), \quad t \geq 0, \; b \in \mathbb{R}_+,$$
is an homogeneous Poisson process with intensity one ($N([0,t]\times[0,b])$ is a Poisson random variable with intensity $ b t$ for any $(t,b) \in \mathbb X$). We set $\mathbb F^N:=(\F_t^N)_{t\geq 0}$ the natural history of $N$, that is $\mathcal{F}_t^N:=\sigma\{N( \mathcal T  \times B), \; \mathcal T \subset \mathcal{B}([0,t]), \; B \in \mathcal{B}(\real_+)\}$. The expectation with respect to $\P$ is denoted by $\E[\cdot]$. We also set $\mathcal{F}_\infty^N:=\lim_{t \to +\infty} \mathcal{F}_t^N$.

\subsubsection{Add-points operators and the Malliavin derivative}

We introduce some elements of Malliavin calculus on Poisson processes.\\
\noindent
For $n\in \mathbb N^*$, $f:\mathbb X^n \to \mathbb R$ is symmetric if for any permutation $\sigma$ on $\{1,\ldots,n\}$, and for any $(x_1,\ldots,x_n) \in \mathbb X^n$, $f(x_1,\ldots,x_n) = f(x_{\sigma(1)},\ldots,x_{\sigma(n)})$.
We set :
$$ L^0(\Omega):=\left\{ F:\Omega \to \real, \; \mathcal{F}_\infty^N-\textrm{ measurable}\right\},$$
$$ L^2(\Omega):=\left\{ F \in L^0(\Omega), \; \E[|F|^2] <+\infty\right\}.$$
Let for $j\in \mathbb{N}^*$
\begin{equation}
\label{definition:L2j}
L^2(\mathbb X^j) := \left\{f:\mathbb{X}^j \to \real, \; \int_{\mathbb{X}^j} |f(x_1,\cdots,x_j)|^2 dx_1 \cdots dx_j <+\infty\right\},
\end{equation}
and
\begin{equation}
\label{definition:symm2}
L^2_s(\mathbb X^j) := \left\{f \textrm{ symmetric and } f \in L^2(\mathbb X^j) \right\}
\end{equation}
the set of symmetric square integrable functions $f$ on $\mathbb{X}^j$. \\\\
\noindent
For $h \in L^2(\mathbb X)$ and $j\geq 1$ we set $h^{\otimes j} \in L_s^2(\mathbb X^j)$ defined as : 
\begin{equation}
\label{eq:otimes}
h^{\otimes j}(x_1,\ldots,x_j) := \prod_{i=1}^j h(x_i), \quad (x_1,\ldots,x_j)\in \mathbb X^j.
\end{equation}
\noindent
The main ingredient we will make use of are the add-points operators on the Poisson space $\Omega$. 

\begin{definition}$[$Add-points operators$]$\label{definitin:shifts}
\begin{itemize}
\item[(i)]
For $k$ in $\mathbb N^*$, and any subset of $\mathbb X$ of cardinal $k$ denoted $\{x_i, \; i\in \{1,\ldots,k\}\} \subset \mathbb X$, we set the measurable mapping :
\begin{eqnarray*}
\varepsilon_{(x_1,\ldots,x_k)}^{+,k} : \Omega & \longrightarrow & \Omega \\
     \omega & \longmapsto   & \omega + \sum_{i=1}^k \delta_{x_i};
\end{eqnarray*}
with the convention that given a representation of $\omega$ as $\omega=\sum_{i=1}^{n} \delta_{y_i}$ (for some $n\in \mathbb N^*$, $y_i \in \mathbb X$), $\omega + \sum_{i=1}^k \delta_{x_i}$ is understood as follows\footnote{Note that given fixed atoms $(x_1,\ldots,x_n)$, as $\P$ is the Poisson measure on $\Omega$, with $\P$-probability one, marks $x_i$ do not belong to the representation of $\omega$.} :
\begin{equation}
\label{eq:addjumpsum}
\omega + \sum_{i=1}^k \delta_{x_i} :=  \sum_{i=1}^{n} \delta_{y_i} + \sum_{i=1}^k \delta_{x_i} \ind{x_i \neq y_i}.
\end{equation}
\item[(ii)] When $k=1$ we simply write $\varepsilon_{x_1}^{+}:=\varepsilon_{x_1}^{+,1}$.
\end{itemize}
\end{definition}
\noindent We now define the Malliavin derivative operator. 

\begin{definition}
\label{definition:Dn}
For $F$ in $L^2(\Omega)$, $n\in \mathbb N^*$, and $(x_1,\ldots,x_n) \in \mathbb X^n$, we set 
\begin{equation}
\label{eq:Dn}
D_{(x_1,\ldots,x_n)} F:= F\circ \varepsilon_{(x_1,\ldots,x_n)}^{+,n} - F.
\end{equation}
For instance when $n=1$, we write $D_x F := D_x^1 F$ which is the difference operator (also called add-one cost operator\footnote{see \cite[p.~5]{Last2016}}). Note that with this definition, for any $\omega$ in $\Omega$, the mapping 
$$ (x_1,\ldots,x_n) \mapsto D_{(x_1,\ldots,x_n)}^n F (\omega) $$
is symmetric and belongs to $L^2_s(\mathbb X^j)$ defined as (\ref{definition:symm2}). 
\end{definition}

\noindent We extend this definition to the iterated Malliavin derivatives.

\begin{prop}(See \textit{e.g.} \cite[Relation (15)]{Last2016})
\label{prop:Dnaltenative}
Let $F$ in $L^2(\Omega)$, $n\in \mathbb N^*$, and $(x_1,\ldots,x_n) \in \mathbb X^n$. We set the $n$th iterated Malliavin derivative operator $D^n$ as 
$$ D^n F = D (D^{n-1} F), \quad n\geq 1; \quad D^0 F:=F.$$
It holds that 
$$ D^n_{(x_1,\ldots,x_n)} F (\omega)= \sum_{J\subset \{1,\cdots,n\}} (-1)^{n-|J|} F\left(\omega + \sum_{j\in J} \delta_{x_j}\right), \quad \textrm{ for a.e. } \omega \in \Omega,$$
where the sum stands for all the subsets $J$ of $\{1,\cdots,n\}$ and $|J|$ denotes the cardinal of $J$.
\end{prop}

\begin{remark}
Note that $\omega + \sum_{j\in J} \delta_{x_j}$ is understood according to (\ref{eq:addjumpsum}).
\end{remark}

\subsubsection{Iterated integrals and the Chaotic expansion}

The decompositions we are going to deal with take the form of iterated stochastic integrals whose definition is made precise in this section.

\begin{notation}
\label{notation:simplex}
For $j\in \mathbb N^*$, $T>0$ and $M>0$ we set :
$$\Delta_j:=\left\{(x_1,\cdots,x_j) \in \mathbb X^j, \; x_i \neq x_k, \; \forall i\neq k \in \{1,\cdots,j\}\right\}.$$
\end{notation}

\begin{definition}
\label{definition:interated}
Let $j \in \mathbb{N}^*$ and $f_j$ an element of $L_s^2(\mathbb X^j)$. We set $I_j(f_j)$ the $j$th iterated integral of $f_j$ against the compensated Poisson measure defined as : 
\begin{align}
\label{eq:In}
& \hspace{-2em} I_j(f_j) \nonumber \\
&\hspace{-2em}:= \int_{\Delta_j}  f_j(x_1,\ldots,x_{j}) (N(dx_{1})-dx_1) \cdots (N(dx_j)-dx_j) \nonumber \\
&\hspace{-2em}=j! \int_{\mathbb X} \int_{[0,t_{j-1})\times \real_+} \cdots \int_{[0,t_{2})\times \real_+} f_j(x_1,\ldots,x_{j}) (N(dx_{1})-dx_1) \cdots (N(dx_j)-dx_j) \nonumber \\
&\hspace{-2em}=j! \int_{\mathbb X} \int_{[0,t_{j-1})\times \real_+} \cdots \int_{[0,t_{2})\times \real_+} f_j((t_1,\theta_1),\ldots,(t_j,\theta_j)) (N(d t_1,d\theta_1)-dt_1 d\theta_1) \cdots (N(d t_j,d\theta_j)-dt_j d\theta_j)
\end{align} 
where we recall the notation $x_i=(t_i,\theta_i)$ and $dx_i=dt_i \, d\theta_i$.
Recall that all the integrals are defined pathwise.
\end{definition}

\noindent
These iterated integrals naturally appear in the chaotic expansion recalled below.

\begin{theorem}[See \textit{e.g.} Theorem 2 in \cite{Last2016}]
\label{th:chaoticgeneral}
Let $F$ in $L^2(\Omega)$. Then 
$$ F = \E[F] + \sum_{j=1}^{+\infty} \frac{1}{j!} I_j(f_j^F),$$
where the convergence of the series holds in $L^2(\omega,\P)$ and where the coefficients $f_j^F$ are the elements of $L_s^2(\mathbb X^j)$ (see (\ref{definition:symm2})) given as 
\begin{eqnarray*}
f_j^F : \mathbb X^j & \longrightarrow & \real_+ \\
      (x_1,\ldots,x_j) & \longmapsto   & \E\left[D_{(x_1,\cdots,x_j)}^j F\right].
\end{eqnarray*}
In addition the decomposition is unique in the sense that : if there exist elements $(g_j)_{j\geq 1}$ with $g_j \in L_s^2(\mathbb X^j)$ such that 
$$ F =  \E[F] + \sum_{j=1}^{+\infty} \frac{1}{j!} I_j(g_j)$$
then $g_j = f_j^F, \; dx-a.e., \; \forall j\geq 1.$
\end{theorem}
\noindent
This decomposition is similar to the Wiener-It\^o decomposition on Gaussian spaces. We conclude this section by recalling the link between the iterated Malliavin derivative and the iterated integrals. 

\begin{theorem}
\label{theorem:generalizedIPP}
Let $j\in \mathbb N^*$, $g_j$ in $L^2_s(\mathbb X^j)$ and $F$ in $L^2(\Omega)$. Then : 
$$ \E\left[\int_{\mathbb X^j} g_j(x_1,\ldots,x_j) \, D_{(x_1,\ldots,x_j) }^j F dx_1\ldots dx_j \right] = \E\left[F I_j(g_j)\right].$$
\end{theorem}

\begin{proof}
This result is standard although we could not find a precise reference fitting to our framework so we provide a sketch of the proof. To the Malliavin derivative, one can associate its dual operator named the divergence operator $\delta$ on a subset of the measurable elements $u:\mathbb X\times \Omega \to \real$. More precisely, for any such $u$ in $\textrm{Dom(}\delta\textrm{)}$ (the domain of the operator see for instance \cite{Nualart_Vives_90,Privault_2009,Last2016}), $\delta(u)$ denotes the unique element in $L^2(\Omega)$ such that :
\begin{equation}
\label{eq:IPPgeneralized}
\E[F \delta(u)] = \int_{\mathbb X}\E[u(x) D_x F] dx, \quad \forall F \in L^2(\Omega).
\end{equation}
By uniqueness, $\delta(u)$ coincides with the It\^o stochastic integral in case $u$ is a predictable process which itself is equal to $I_1(u)$ when $u$ is deterministic. 
Similarly, the iterated divergence of order $j$ denoted $\delta^j$ can be defined as the dual operator of the $j$th Malliavin derivative $D^j$. In case of a deterministic element $g_j$, $I_j(g_j) = \delta^j(g_j)$. To see this, note that for any $j \geq 2$ and any $g \in L^2_s(\mathbb X^j)$ it holds that
$$ I_j(g) = \delta(I_{j-1}(g(\cdot,\bullet))), \quad \cdot \in \mathbb X, \; \bullet \in \mathbb X^{j-1}, \; \textrm{ in } L^2(\Omega).$$
Using then the Malliavin integration by parts formula (\ref{eq:IPPgeneralized}) one gets the result by induction.
\end{proof}
\noindent
We conclude this section with the well-known relation between the Malliavin derivatives and the iterated integrals that can be found for example in \cite{Last2016,Privault_2009}.

\begin{prop}
\label{prop:derinIj}
\begin{itemize}
\item[(i)]
Let $j\in \mathbb N^*$, $k\in \mathbb N^*$, with $k\leq j$ and $h \in L^2(\mathbb X)$. Then : 
$$ D_{(x_1,\cdots,x_k)}^k I_j(h^{\otimes j}) = \frac{j!}{(j-k)!} I_{j-k}(h^{\otimes (j-k)}) \prod_{i=1}^k h(x_i), \quad \forall (x_1,\cdots,x_k) \in \mathbb X^k.$$
\item[(ii)] In addition, for $k>j$,
$$ D_{(x_1,\cdots,x_k)}^k I_j(h^{\otimes j}) =0, \quad \forall (x_1,\cdots,x_k) \in \mathbb X^k. $$
\end{itemize}
\end{prop}

\section{Notion of pseudo-chaotic expansion}
\label{section:pseudochaotic}

We now present an alternative decomposition that we name pseudo-chaotic expansion. Recall that according to Theorem \ref{th:chaoticgeneral}, any $F$ in $L^2(\Omega)$ admits a chaotic expansion as : 
$$ F = \E[F] + \sum_{j=1}^{+\infty} \frac{1}{j!} I_j(f_j^F),$$
with $f_j^F(x_1,\cdots,x_j)=\E\left[D_{(x_1,\cdots,x_j)}^j F\right]$.\\\\
\noindent
For technical reasons we will also consider the same property but for a baseline Poisson measure on a given bounded subset of $\mathbb X$.

\begin{definition}
\label{definition:truncatedstuff}
For $(T,M)\in \mathbb X$ we set : 
\begin{itemize}
\item[(i)] $ \Delta^{T,M}_j:=\left\{(x_1,\cdots,x_j) \in ([0,T]\times[0,M])^j, \; x_i \neq x_k, \; \forall i\neq k \in \{1,\cdots,j\}\right\},$ 
\item[(ii)] $N^{T,M}$ the truncated Poisson measure $N^{T,M}$ defined as : 
$$ N^{T,M}(A):=\int_{A} \ind{[0,T]\times[0,M]}(x) N(dx), \quad A \in \mathcal B(\mathbb X). $$
\item[(iii)] $L^{2,T,M}(\Omega)$ the set of random variables $F$ in $L^2(\Omega)$ such that there exists $(f_j)_{j\geq 1}$ with $f_j \in L_s^2(([0,T]\times[0,M])^j)$ such that 
$$F= \E[F] + \sum_{j=1}^{+\infty} \frac{1}{j!} I_j(f_j),$$
that is the set of random variables admitting the chaotic expansion with $N$ replaced by $N^{T,M}$.
\end{itemize}
\end{definition}

\begin{definition}$[$Pseudo-chaotic expansion$]$
\label{eq:pseudo}
\begin{enumerate}
\item A random variable $F$ in $L^2(\Omega)$ is said to have a pseudo-chaotic expansion with respect to the counting process $N$ if there exists $(g_j)_{j \geq 1}$, $g_j \in L^2_s(\mathbb X^j)$ for all $j\in \mathbb N^*$ such that : 
\begin{equation}
\label{eq:definitionpseudochaotic}
F =\sum_{j=1}^{+\infty} \frac{1}{j!} \int_{\mathbb X^j} g_j(x_1,\ldots,x_j) N(dx_1) \cdots N(dx_j),
\end{equation}
where the series converges in $L^2(\Omega)$. 
\item Fix $(T,M) \in \mathbb X$. A random variable $F$ in $L^{2,T,M}(\Omega)$ is said to have a pseudo-chaotic expansion with respect to the counting process $N$ if there exists $(g_j)_{j\geq 1}$, $g_j \in L^2_s(([0,T]\times[0,M])^j)$ for all $j\in \mathbb N^*$ such that : 
\begin{align}
\label{eq:psueudochaoticTM}
F &=\sum_{j=1}^{+\infty} \frac{1}{j!} \int_{\mathbb X^j} g_j(x_1,\ldots,x_j) N^{T,M}(dx_1) \cdots N^{T,M}(dx_j)\nonumber \\
&=\sum_{j=1}^{+\infty} \frac{1}{j!} \int_{([0,T]\times[0,M])^j} g_j(x_1,\ldots,x_j) N(dx_1) \cdots N(dx_j)
\end{align}
where the series converges in $L^2(\Omega)$.
\end{enumerate}
\end{definition}

\begin{remarks}\mbox{}
\begin{itemize}
\item[-] Recall that with the notations here  above (and Notation \ref{notation:simplex}), the symmetry of functions $g_j$ entails that
\begin{align*}
& \int_{\mathbb X^j} g_j(x_1,\ldots,x_j) N(dx_j) \cdots N(dx_1) \\
&= \int_{\Delta_j} g_j(x_1,\ldots,x_j) N(dx_1) \cdots N(dx_j) \\
&= \int_{\Delta_j} g_j((t_1,\theta_1),\ldots,(t_j,\theta_j)) N(dt_1,d\theta_1) \cdots N(dt_j,d\theta_j) \\
&= j! \int_{\mathbb X} \int_{[0,t_{j-1})\times \real_+} \cdots \int_{[0,t_{2})\times \real_+} g_j((t_1,\theta_1),\ldots,(t_j,\theta_j)) N(dt_1,d\theta_1) \cdots N(dt_j,d\theta_j).
\end{align*} 
\item[-] Note that in each term $\int_{\mathbb X^j} g_j(x_1,\ldots,x_j) N(dx_1) \cdots N(dx_j)$ in (\ref{eq:definitionpseudochaotic}), the multiple integration coincides with the one with respect to the so-called factorial measures presented in \cite[Appendix]{Last2016}.\\
\end{itemize}
\end{remarks}

\begin{definition}
\label{eq:setofpseudo}
We set
$$\mathcal P:=\{F \in L^2(\Omega), \textrm{ which admits a pseudo-chaotic expansion with respect to }N\},$$
and for $(T,M) \in \mathbb X$,
$$\mathcal P^{T,M}:=\{F \in L^2(\Omega), \textrm{ which admits a pseudo-chaotic expansion with respect to }N^{T,M}\}.$$
\end{definition}

\noindent
Before studying and characterizing those random variables which admit a pseudo-chaotic expansion we need some preliminary results collected in the section below.

\subsection{Some preliminary results}

\begin{definitionprop}
\label{definition:truncatedstuffbis}
For $(T,M)\in \mathbb X$.  
\begin{itemize}
\item[(i)] Let the random variable $L^{T,M}$ on $\Omega$,
\begin{equation}
\label{eq:L}
L^{T,M}(\omega) := \exp\left(MT\right) \ind{N([0,T]\times[0,M])(\omega)=0} = \exp\left(MT\right) \ind{\omega([0,T]\times[0,M])=0}, \quad \omega \in \Omega.
\end{equation}
It holds that 
$$ L^{T,M} = 1 +\sum_{j=1}^{+\infty} \frac{1}{j!} I_j((-\ind{[0,T]\times[0,M]})^{\otimes j}),$$
and 
\begin{equation}
\label{eq:derivL}
D_{(x_1,\cdots,x_j)}^j L^{T,M} = L^{T,M} \prod_{i=1}^j (-\ind{[0,T]\times [0,M]}(x_i)), \quad \forall (x_1,\cdots,x_j) \in \mathbb X^j.
\end{equation}
\item[(ii)] $\mathbb Q^{T,M}$ defined on $(\Omega,\mathcal F_T^N)$ as $\frac{ d \mathbb Q^{T,M}}{d\P} =:L^{T,M}$ is a probability measure.\\\\
\noindent
As the support of $\mathbb Q^{T,M}$ is contained in $\{N([0,T]\times[0,M])=0\}$ we name $\mathbb Q^{T,M}$ the vanishing Poisson measure as it brings the intensity of $N$ to $0$ on the rectangle $[0,T]\times[0,M]$ and
$$ \E^{\mathbb Q^{T,M}}[N^{T,M}(A)] = 0, \quad \forall A \in \mathcal B(\mathbb X). $$
\end{itemize}
\end{definitionprop}

\begin{proof}
Set 
$$ L= 1 +\sum_{j=1}^{+\infty} \frac{1}{j!} I_j((-\ind{[0,T]\times[0,M]})^{\otimes j}).  $$
Following \cite[Proposition 6.3.1]{Privault_2009}, this expression is the chaotic expansion of the stochastic exponential at time $T$ of the deterministic function $x\mapsto -\ind{(0,T)\times[0,M]}(x) $ against the compensated Poisson measure, that is 
\begin{align*}
L &= \exp\left(\int_{\mathbb X} -\ind{[0,T]\times [0,M]}(x) \tilde{N}(dx)\right) \prod_{x, \; N(\{x\})=1} \left[(1-\ind{[0,T]\times [0,M]}(x)) \exp\left(\ind{[0,T]\times [0,M]}(x)\right)\right] \\
&= \exp\left(M T \right) \ind{N([0,T]\times[0,M])=0}\\
&= L^{T,M}.
\end{align*}
In addition, by definition of the Malliavin derivative $D$ and the exponential structure of $L$ we have that 
$$D_{x_1} L = -L \ind{[0,T]\times [0,M]}(x_1).$$
Relation (\ref{eq:derivL}) then results from the fact that $-\ind{[0,T]\times [0,M]}$ is deterministic and that $D^j = D^{j-1} D$.\\
\noindent
Finally, as $\E[L^{T,M} = 1]$, the measure $\mathbb Q^{T,M}$ in (ii) is well defined and  is  a probability measure (not equivalent to $\P$).
\end{proof}

\begin{remark}
Our analysis will be based on the intervention of the quantity $L^{T,M}$ or of the Poisson vanishing measure $\mathbb Q^{T,M}$ which is properly defined only on bounded subsets of $\mathbb X$; which explains why we derive our results for the truncated Poisson measures $N^{T,M}$ and not for $N$.
\end{remark}

\begin{prop}
\label{prop:degatsQTM}
Let $(T,M)\in \mathbb X$, $j\in \mathbb N^*$, $f\in L_s^2(([0,T]\times[0,M])^j)$. Let $F$ of the form (\ref{eq:psueudochaoticTM}) (that is $F \in \mathcal P^{T,M}$).
Then 
$$ \E^{\mathbb Q^{T,M}}\left[F\right] = \E\left[L^{T,M} F\right] =0.$$
\end{prop}

\begin{proof}
The result is an immediate consequence of the definition of $\mathbb Q^{T,M}$ which is supported on the set $N([0,T]\times[0,M])=0$ and of the form of $F$ as a sum (\ref{eq:psueudochaoticTM}) involving only integrals against $N$ on $[0,T]\times[0,M]$.
\end{proof}

\noindent
Our first main result which gives a motivation to the definition of the notion of the pseudo-chaotic expansion lies in the theorem below.

\begin{theorem}
\label{th:countingexpectation}
Let $T,M>0$ and $F$ in $L^{2,T,M}(\Omega)$ (see (ii) of Definition \ref{definition:truncatedstuff}). 
Then
$$ \sum_{j=1}^{+\infty} \frac{(-1)^j}{j!} \int_{\mathbb X^j} \E\left[D_{(x_1,\cdots,x_j)}^j F\right] dx_1 \cdots dx_j 
=\E\left[F(L^{T,M}-1)\right].$$
\end{theorem}

\begin{proof}
Using the Malliavin integration by parts it holds that : 
\begin{align*}
&\sum_{j=1}^{+\infty} \frac{(-1)^j}{j!} \int_{\mathbb X^j} \E\left[D_{(x_1,\cdots,x_j)}^j F\right] dx_1 \cdots dx_j \\
&\sum_{j=1}^{+\infty} \frac{(-1)^j}{j!} \int_{([0,T\times [0,M])^j} \E\left[D_{(x_1,\cdots,x_j)}^j F\right] dx_1 \cdots dx_j \\
&=\sum_{j=1}^{+\infty} \frac{(-1)^j}{j!} \E\left[ \int_{([0,T\times [0,M])^j} D_{(x_1,\cdots,x_j)}^j F \; dx_1 \cdots dx_j\right]\\
&=\sum_{j=1}^{+\infty} \frac{(-1)^j}{j!} \E\left[ \int_{\mathbb X^j} \prod_{i=1}^j \ind{[0,T]\times [0,M]}(x_i) D_{(x_1,\cdots,x_j)}^j F \; dx_1 \cdots dx_j\right]\\
&=\sum_{j=1}^{+\infty} \frac{1}{j!} \E\left[\int_{\mathbb X^j} \prod_{i=1}^j (-\ind{[0,T]\times [0,M]}(x_i)) D_{(x_1,\cdots,x_j)}^j F dx_1 \cdots dx_j\right]\\
&=\sum_{j=1}^{+\infty} \frac{1}{j!} \E\left[\int_{\mathbb X^j} (-\ind{[0,T]\times [0,M]})^{\otimes j})(x_1,\cdots,x_j) D_{(x_1,\cdots,x_j)}^j F dx_1 \cdots dx_j\right]\\
&=\sum_{j=1}^{+\infty} \frac{1}{j!} \E\left[F \; I_j\left((-\ind{[0,T]\times [0,M]})^{\otimes j}\right)\right], \quad \textrm{ by Theorem \ref{theorem:generalizedIPP}}\\
&=\E\left[F (L^{T,M}-1)\right].
\end{align*}
\end{proof}

\noindent The previous result will find interest for instance when $F=H_T$ for $H$ a counting process and $T>0$. We make precise the definition of such processes. 

\begin{definition}$[$Counting process with bounded intensity by Poisson imbedding$]$\\
Let $\lambda$ be a $\mathcal F^N$-predictable process such that 
$$\exists M>0, \textrm{ such that }\; \lambda_t \leq M, \quad \forall t\geq 0, \; \P-a.s..$$ 
The process $H$ defined below is a counting process with intensity $\lambda$ :
\begin{equation}
\label{eq:counting}
H_t = \int_{(0,t]\times \real_+} \ind{\theta \leq \lambda_s} N(ds,d\theta)
= \int_{(0,t]\times [0,M]} \ind{\theta \leq \lambda_s} N(ds,d\theta), \quad t\geq 0.
\end{equation}
\\\\\noindent
Using the chaotic expansion we have for any $T>0$ that : 
\begin{equation}
\label{eq:chaostemp}
H_T = \E[H_T] + \sum_{j\geq 1} \frac{1}{j!} I_j(f_j^{H_T}), \quad \textrm{ with } \quad  f_j^{H_T}(x_1,\cdots,x_j) = \E\left[D_{(x_1,\cdots,x_j)}^j H_T\right], \; x_i \in [0,T] \times [0,M].
\end{equation}
\end{definition}

\noindent
We now apply Theorem \ref{th:countingexpectation} to a counting process. 

\begin{corollary}
\label{cor:expectcounting}
Let $T>0$, $H$ a counting process with bounded intensity $\lambda$ by some $M>0$ so that $H_T$ is given by (\ref{eq:counting}). Then 
\begin{align*}
&\sum_{j=1}^{+\infty} \frac{(-1)^j}{j!} \int_{\mathbb X^j} f_j^{H_T}(x_1,\cdots,x_j) dx_1 \cdots dx_j \\
&= \sum_{j=1}^{+\infty} \frac{(-1)^j}{j!} \int_{\mathbb X^j} \E\left[D_{(x_1,\cdots,x_j)}^j H_T\right] dx_1 \cdots dx_j \\
&=-\E[H_T].
\end{align*}
\end{corollary}

\begin{proof}
By Theorem \ref{th:countingexpectation},
\begin{align*}
&\sum_{j=1}^{+\infty} \frac{(-1)^j}{j!} \int_{([0,T]\times[0,M])^j} f_j^{H_T}(x_1,\cdots,x_j) dx_1 \cdots dx_j \\
&= \sum_{j=1}^{+\infty}  \frac{(-1)^j}{j!}  \int_{([0,T]\times[0,M])^j} \E\left[D_{(x_1,\cdots,x_j)}^j H_T\right] dx_1 \cdots dx_j \\
&=\E\left[H_T (L^{T,M}-1)\right].
\end{align*}
Proposition \ref{prop:degatsQTM} entails then that
$$\E\left[H_T L^{T,M}\right]=\exp(MT) \E\left[H_T \ind{N([0,T]\times[0,M])=0}\right]=0$$
which concludes the proof.
\end{proof}

\begin{remark}
\label{rk:core}
The previous result is at the core of our analysis. This means that for a counting process with bounded intensity, the only term with only Lebesgue integrals ($dx$) in Expansion (\ref{eq:chaostemp}) vanishes with $\E[H_T]$. In other words, all the terms in Expansion (\ref{eq:chaostemp}) involves at least on integral against $N$.
\end{remark}

\noindent
We conclude this section with  a generalized version of Theorem \ref{th:countingexpectation}.

\begin{lemma}
\label{lemma:technicalsumderivatives}
Fix $(T,M) \in \mathbb X$ and $F$ in $L^{2,T,M}(\Omega)$. For any $k\in \mathbb N^*$ and any $(x_1,\cdots,x_k) \in ([0,T]\times[0,M])^k$, it holds that :
\begin{align*}
&\E\left[D_{(x_1,\ldots,x_k)}^k F\right] + \sum_{j=k+1}^{+\infty} \frac{(-1)^{j-k}}{(j-k)!} \int_{([0,T]\times[0,M])^{j-k}} \E\left[D_{(x_1,\ldots,x_k,x_{k+1},\ldots,x_{j})}^j F\right] dx_j \cdots dx_{k+1} \\
&= \E\left[ D_{(x_1,\ldots,x_k)}^k F L^{T,M} \right].
\end{align*}
\end{lemma}

\begin{proof}
The property $D^j = D^{j-k} D^k$ (for the first equality) and Theorem \ref{theorem:generalizedIPP}  applied to $D_{(x_1,\ldots,x_k)}^k F$ (for the second equality) imply
\begin{align*}
& \E\left[D_{(x_1,\ldots,x_k)}^k F\right] + \sum_{j=k+1}^{+\infty} \frac{(-1)^{j-k}}{j!} \frac{j!}{(j-k)!} \int_{([0,T]\times[0,M])^{j-k}} \E\left[D_{(x_1,\ldots,x_k,x_{k+1},\ldots,x_{j})}^j F\right] dx_j \cdots dx_{k+1} \\
&= \E\left[D_{(x_1,\ldots,x_k)}^k F\right] + \sum_{j=k+1}^{+\infty} \frac{(-1)^{j-k}}{j!} \frac{j!}{(j-k)!} \int_{([0,T]\times[0,M])^{j-k}} \E\left[D_{(x_{k+1},\ldots,x_{j})}^{j-k} D_{(x_1,\ldots,x_k)}^k F\right] dx_j.\\
&= \E\left[D_{(x_1,\ldots,x_k)}^k F\right] + \sum_{j=k+1}^{+\infty} \frac{1}{j!} \frac{j!}{(j-k)!} \E\left[I_{j-k}\left((-\ind{[0,T]\times [0,M]})^{\otimes (j-k)}\right) D_{(x_1,\ldots,x_k)}^k F\right] \\
&= \E\left[D_{(x_1,\ldots,x_k)}^k F\right] + \E\left[ \sum_{j=k+1}^{+\infty} \frac{1}{j!} \frac{j!}{(j-k)!} I_{j-k}\left((-\ind{[0,T]\times [0,M]})^{\otimes (j-k)}\right) D_{(x_1,\ldots,x_k)}^k F\right].
\end{align*}
Then (i)  of Proposition \ref{prop:derinIj} entails that 
\begin{align*}
& \E\left[D_{(x_1,\ldots,x_k)}^k F\right] + \sum_{j=k+1}^{+\infty} \frac{(-1)^{j-k}}{j!} \frac{j!}{(j-k)!} \int_{([0,T]\times[0,M])^{j-k}} \E\left[D_{(x_1,\ldots,x_k,x_{k+1},\ldots,x_{j})}^j F\right] dx_j \cdots dx_{k+1} \\
&= \E\left[D_{(x_1,\ldots,x_k)}^k F\right] + \E\left[ \sum_{j=k+1}^{+\infty} \frac{1}{j!} (-1)^k \left(D_{(x_1,\ldots,x_k)}^k I_{j}\left((-\ind{[0,T]\times [0,M]})^{\otimes j}\right)\right)D_{(x_1,\ldots,x_k)}^k F\right] \\
&=\E\left[D_{(x_1,\ldots,x_k)}^k F\right] + (-1)^k \E\left[ D_{(x_1,\ldots,x_k)}^k F \left(D_{(x_1,\ldots,x_k)}^k \sum_{j=k+1}^{+\infty} \frac{1}{j!} I_{j}\left((-\ind{[0,T]\times [0,M]})^{\otimes j}\right) \right)\right] \\
&= \E\left[D_{(x_1,\ldots,x_k)}^k F\right] + (-1)^k\E\left[ D_{(x_1,\ldots,x_k)}^k F \left(D_{(x_1,\ldots,x_k)}^k \left(L^{T,M}-1-\sum_{j=1}^k \frac{1}{j!} I_{j}\left((-\ind{[0,T]\times [0,M]})^{\otimes j}\right) \right)\right)\right],
\end{align*}
where we have used Proposition-Definition \ref{definition:truncatedstuffbis}. Thus (i) and (ii) of Proposition \ref{prop:derinIj} give
\begin{align*}
&\E\left[D_{(x_1,\ldots,x_k)}^k F\right] + \sum_{j=k+1}^{+\infty} \frac{(-1)^{j-k}}{j!} \frac{j!}{(j-k)!} \int_{([0,T]\times[0,M])^{j-k}} \E\left[D_{(x_1,\ldots,x_k,x_{k+1},\ldots,x_{j})}^j F\right] dx_j \cdots dx_{k+1} \\
&= \E\left[D_{(x_1,\ldots,x_k)}^k F\right] + (-1)^k \E\left[ D_{(x_1,\ldots,x_k)}^k F \left(D_{(x_1,\ldots,x_k)}^k L^{T,M} +(-1)^{k+1}\right)\right] \\
&= \E\left[D_{(x_1,\ldots,x_k)}^k F\right] + (-1)^k \E\left[ D_{(x_1,\ldots,x_k)}^k F \left((-1)^{k} L^{T,M} +(-1)^{k+1}\right)\right] \\
&= \E\left[D_{(x_1,\ldots,x_k)}^k F\right] + \E\left[ D_{(x_1,\ldots,x_k)}^k F \left(L^{T,M} -1 \right)\right] \\
&= \E\left[L^{T,M} D_{(x_1,\ldots,x_k)}^k F \right].
\end{align*}
\end{proof}

\subsection{Characterization of $\mathcal P^{T,M}$}

Throughout this section we fix $(T,M)$ in $\mathbb X$. Corollary \ref{cor:expectcounting} suggests  that random variables of the form $H_T$ with $H$ a counting process satisfy the pseudo-chaotic expansion. We make precise this point and characterize the set $\mathcal P^{T,M}$.

\begin{theorem}
\label{th:characpseudo}
An element $F$in $L^2(\Omega)$ admits a pseudo-chaotic expansion with respect to the counting process $N^{T,M}$ (that is $F \in \mathcal P^{T,M}$) if and only if
\begin{equation}
\label{eq:expectationsumagain}
\E[F] = \sum_{j=1}^{+\infty} \frac{ (-1)^{j+1}}{j!} \int_{([0,T]\times[0,M])^j} \E\left[D_{(x_1,\cdots,x_j)}^j F\right] dx_1 \cdots dx_j.
\end{equation}
In that case the pseudo-chaotic expansion of $F$ is given by 
\begin{equation}
\label{eq:pseudodecompositionck}
F = \sum_{k=1}^{+\infty} \int_{([0,T]\times[0,M])^k} \frac{1}{k!} c_k(x_1,\ldots,x_k) N(dx_1) \cdots N(dx_k),
\end{equation}
with 
\begin{equation}
\label{eq:ck}
c_k(x_1,\ldots,x_k) := \E\left[ L^{T,M} \, D_{(x_1,\ldots,x_k)}^k F \right] = \E^{\mathbb{Q}^{T,M}}\left[D_{(x_1,\ldots,x_k)}^k F \right], \quad \forall (x_1,\ldots,x_k)\in ([0,T]\times[0,M])^k.
\end{equation}
\end{theorem}

\begin{proof}
Let $F$ in $\mathcal P^{T,M}$. Then, according to Theorem \ref{th:countingexpectation} and Proposition \ref{prop:degatsQTM}
\begin{align*}
&\sum_{j=1}^{+\infty} \frac{(-1)^j}{j!} \int_{([0,T]\times[0,M])^j} \E\left[D_{(x_1,\cdots,x_j)}^j F\right] dx_1 \cdots dx_j\\
&=\E\left[F (L^{T,M}-1)\right]\\
&=\E^{\mathbb Q^{T,M}}[F] - \E[F]\\
&=-\E[F].
\end{align*}
So 
$$ \sum_{j=1}^{+\infty} \frac{ (-1)^{j+1}}{j!} \int_{([0,T]\times[0,M])^j} \E\left[D_{(x_1,\cdots,x_j)}^j F\right] dx_1 \cdots dx_j = \E\left[F\right].$$
Conversely assume $F\in L^2(\Omega)$ is such that Relation (\ref{eq:expectationsumagain}) is true. The chaotic expansion (see Theorem \ref{th:chaoticgeneral}) allows one to write 
$$ F = \E[F] + \sum_{j=1}^{+\infty} \frac{1}{j!} I_j(\E[D^j F]).$$
The definition of the iterated integrals $I_j$ together with Relation (\ref{eq:expectationsumagain}) implies then that 
\begin{align*}
F &= \E[F] + \sum_{j=1}^{+\infty} \frac{1}{j!} I_j(f_j^F), \quad \textrm{ with } \quad  f_j^F(x_1,\cdots,x_j) = \E\left[D_{(x_1,\cdots,x_j)}^j F\right] \\
&= \sum_{k=1}^{+\infty} \frac{1}{k!} \int_{([0,T]\times[0,M])^k} c_k(x_1,\ldots,x_k) N(dx_1) \cdots N(dx_k),
\end{align*}
with 
\begin{align*}
c_k(x_1,\ldots,x_k)&:= f_k^F(x_1,\ldots,x_k) \\
&+ {k!} \sum_{j=k+1}^{+\infty} \frac{(-1)^{j-k}}{j!} \frac{j!}{{k!}(j-k)!}\int_{([0,T]\times[0,M])^{j-k}} f_j^F(x_1,\ldots,x_k,x_{k+1},\ldots,x_{j}) dx_{k+1} \cdots dx_j.
\end{align*}
Here the number $\frac{j!}{{k!}(j-k)!}$  of $k$-combinations among $j$ choices denotes the number of times the integral $\int_{([0,T]\times[0,M])^{j-k}} f_j^F(x_1,\ldots,x_k,x_{k+1},\ldots,x_{j}) dx_{k+1} \cdots dx_j$ of the symmetric function $f_j^F$ appears in the expansion of $I_j(f_j^F)$. Note also the choice of normalisation by factoring $c_k$ with $\frac{1}{k!}$ which explains the $k!$ factor in front of the sum.
We now compute each of these terms. Using the definition of the $f_j^F$ functions and Lemma \ref{lemma:technicalsumderivatives} we get 
\begin{align*}
&c_k(x_1,\ldots,x_k)\\
&= f_k^F(x_1,\ldots,x_k) + \sum_{j\geq k+1} \frac{(-1)^{j-k}}{j!} \frac{j!}{(j-k)!}\int_{\mathbb X^{j-k}} f_j^F(x_1,\ldots,x_k,x_{k+1},\ldots,x_{j}) dx_{k+1} \cdots dx_j\\
&= \E\left[D_{(x_1,\ldots,x_k)}^k F\right] + \sum_{j\geq k+1} \frac{(-1)^{j-k}}{j!} \frac{j!}{(j-k)!} \int_{\mathbb X^{j-k}} \E\left[D_{(x_1,\ldots,x_k,x_{k+1},\ldots,x_{j})}^j F\right] dx_{k+1} \cdots dx_j \\
&= \E\left[L^{T,M} D_{(x_1,\ldots,x_k)}^k F \right].
\end{align*}
\end{proof}

\begin{remark}
Note that the uniqueness of the coefficients in the chaotic expansion transfers to the uniqueness of the pseudo-chaotic expansion when it exists and is given by the coefficients $c_k$ in (\ref{eq:ck}).
\end{remark}

\noindent
We now apply this result to counting processes with bounded intensity processes.

\begin{theorem}
\label{th:pseudochaoticcounting}
Let $T>0$, $H$ a counting process with bounded intensity $\lambda$ by $M>0$ so that $H_T$ is given by (\ref{eq:counting}). Then $H_T$ admits a pseudo-chaotic expansion with respect to $N^{T,M}$ with 
\begin{equation}
\label{eq:pseudodecompositionckcounting}
H_T = \sum_{k=1}^{+\infty} \int_{([0,T]\times[0,M])^k} \frac{1}{k!} c_k(x_1,\ldots,x_k) N(dx_1) \cdots N(dx_k),
\end{equation}
\begin{equation}
\label{eq:ckcounting}
\hspace*{-0.5cm}c_k(x_1,\ldots,x_k) := \E\left[L^{T,M} D_{(x_{(1)},\ldots,x_{(k-1)})}^{k-1} \ind{\theta_{(k)}\leq \lambda_{(t_k)}}\right], \quad \forall (x_1,\ldots,x_k)\in ([0,T]\times[0,M])^k
\end{equation}
where according to Notation \ref{notation:ordered},  \;
$0 \leq t_{(1)} \leq \cdots \leq t_{(k)} \leq T$ \; are the ordered elements $(t_1,\ldots,t_k)$ and $x_{(i)}:=(t_{(i)},\theta_{(i)})$. 
\end{theorem}

\begin{proof}
Theorem \ref{cor:expectcounting} and Theorem \ref{th:characpseudo} give that $H_T$ admits a pseudo-chaotic expansion and 
$$H_T = \sum_{k=1}^{+\infty} \int_{([0,T]\times[0,M])^k} \frac{1}{k!} c_k(x_1,\ldots,x_k) N(dx_1) \cdots N(dx_k),$$
with $c_k(x_1,\ldots,x_k) := \E\left[L^{T,M} D_{(x_1,\ldots,x_k)}^k H_T  \right], \quad \forall (x_1,\ldots,x_k)\in ([0,T]\times[0,M])^k.$
Let $k\geq 1$ and $(x_1,\ldots,x_k)$ in $([0,T]\times[0,M])^k$. 
As $(x_1,\cdots,x_k) \mapsto D_{(x_1,\ldots,x_k)}^k H_T$ is symmetric, 
$$D_{(x_1,\ldots,x_k)}^k H_T = D_{(x_{(1)},\ldots,x_{(k)})}^k H_T =D_{x_1} \cdots D_{x_k} H_T.$$ 
Using the definition of $D$ and using the fact that 
$$ H_T = \int_{(0,T]\times [0,M]} \ind{\theta \leq \lambda_s} N(ds,d\theta) $$
one gets that 
\begin{equation}
\label{eq:DH}
D_{(x_{(1)},\ldots,x_{(k)})}^k H_T = D_{(x_{(1)},\ldots,x_{(k-1)})}^{k-1} \ind{\theta_{(k)}\leq \lambda_{(t_k)}} + \int_{[0,T]\times\real_+}  D_{(x_{(1)},\ldots,x_{(k)})}^k  \ind{\theta\leq \lambda_t} N(dt,d\theta).
\end{equation}
As $L^{T,M}$ annihilates the Poisson process on $[0,T]\times[0,M]$ it holds that 
$$ \E\left[L^{T,M} D_{(x_1,\ldots,x_k)}^k H_T  \right] = \E\left[L^{T,M} D_{(x_{(1)},\ldots,x_{(k-1)})}^{k-1} \ind{\theta_{(k)}\leq \lambda_{(t_k)}}\right].$$
\end{proof}

\section{Application to linear Hawkes processes}
\label{section:Hawkes}

Throughout this section $\Phi:\real_+\to\real_+$ denotes a map in $L^1(\real_+;dt)$.

\subsection{Generalities on linear Hawkes processes}

\begin{assumption}
\label{assumption:Phi}
The mapping $\Phi : \real_+ \to \real_+$ belongs to $L^1(\real_+;dt)$ with 
$$\|\Phi\|_1:=\int_{\real_+} \Phi(t) dt < 1.$$
\end{assumption}
\noindent For $f,g$ in $L^1(\real_+;dt)$ we define the convolution of $f$ and $g$ by 
$$(f\ast g)(t):=\int_0^t f(t-u) g(u) du, \quad t \geq 0.$$

\begin{prop}[See \textit{e.g.} \cite{Bacryetal2013}]
\label{prop:Phin}
Assume $\Phi$ enjoys Assumption \ref{assumption:Phi}. Let 
\begin{equation}
\label{eq:Phin}
\Phi_1:=\Phi, \quad \Phi_n(t):=\int_0^t \Phi(t-s) \Phi_{n-1}(s) ds, \quad t \in \real_+, \; n\in \mathbb{N}^*.
\end{equation}
For every $n\geq 1$, $\|\Phi_n\|_1 = \|\Phi\|_1^n$ and the mapping $\Psi:=\sum_{n=1}^{+\infty} \Phi_n$ is well-defined as a limit in $L_1(\real_+;dt)$ and $\|\Psi\|_1 = \frac{\|\Phi\|_1}{1-\|\Phi\|_1}$.
\end{prop}

\begin{definition}[Linear Hawkes process, \cite{Hawkes}]
\label{def:standardHawkes}
Let $(\Omega,\mathcal F,\P,\mathbb F:=(\mathcal F_t)_{t\geq 0})$ be a filtered probability space, $\mu>0$ and $\Phi:\real_+ \to \real_+$ satisfying Assumption \ref{assumption:Phi}. A linear Hawkes process $H:=(H_t)_{t\geq 0}$ with parameters $\mu$ and $\Phi$ is a counting process such that   
\begin{itemize}
\item[(i)] $H_0=0,\quad \P-a.s.$,
\item[(ii)] its ($\mathbb{F}$-predictable) intensity process is given by
$$\lambda_t:=\mu + \int_{(0,t)} \Phi(t-s) dH_s, \quad t\geq 0,$$
that is for any $0\leq s \leq t$ and $A \in \mathcal{F}_s$,
$$ \E\left[\textbf{1}_A (H_t-H_s) \right] = \E\left[\int_{(s,t]} \textbf{1}_A \lambda_r dr \right].$$
\end{itemize}
\end{definition}

\subsection{Pseudo-chaotic expansion of linear Hawkes processes and explicit representation}
We aim at providing the coefficients in the pseudo-chaotic expansion of a linear Hawkes process. We start with some general facts regarding a linear Hawkes process.

\begin{prop}
\label{prop:systemHawkesdeterministe}
Let $\Phi$ as in Assumption \ref{assumption:Phi} and $\mu>0$ and $(H,\lambda)$ the Hawkes process defined as the unique solution to the SDE 
$$
\left\lbrace
\begin{array}{l}
H_t = \int_{(0,t]\times \real} \ind{\theta\leq \lambda_s} N(ds,d\theta),\\\\
\lambda_t = \mu +\int_{(0,t)\times\real_+} \Phi(t-s) dH_s,\quad t\geq 0
\end{array}
\right.
$$
Let $n \in \mathbb N^*$ and $\{y_1,\ldots,y_n\} = \{(s_1,\theta_1),\ldots,(s_n,\theta_n)\} \subset \mathbb X$ with $ 0 \leq s_1\leq \cdots \leq s_n \leq t$. \\
We set $(a_1^{\{y_1,\ldots,y_n\}},\cdots,a_n^{\{y_1,\ldots,y_n\}})$ the solution to the system 
\begin{equation}
\label{eq:systemHawkes1}
\left\lbrace
\begin{array}{l}
a_1^{\{y_1,\ldots,y_n\}} = \mu + \ind{\theta_1 \leq \mu},\\\\
a_j^{\{y_1,\ldots,y_n\}} = \mu + \displaystyle{\sum_{i=1}^{j-1} \Phi(s_j-s_i) \ind{\theta_i \leq a_i^{\{y_1,\ldots,y_n\}}}}, \quad j \in \{2,\ldots,k\}.
\end{array}
\right.
\end{equation}
which is the triangular system
\begin{equation}
\label{eq:systemHawkes2}
\left\lbrace
\begin{array}{l}
a_1^{\{y_1,\ldots,y_n\}} = \mu + \ind{\theta_1 \leq \mu},\\\\
a_2^{\{y_1,\ldots,y_n\}} = \mu + \displaystyle{\Phi(s_2-s_1) \ind{\theta_2 \leq a_1^{\{y_1,\ldots,y_n\}}}},\\\\
\hspace{5em}\vdots\\\\
a_n^{\{y_1,\ldots,y_n\}} = \mu + \displaystyle{\sum_{i=1}^{n-1} \Phi(s_n-s_i) \ind{\theta_i \leq a_i^{\{y_1,\ldots,y_n\}}}}.
\end{array}
\right.
\end{equation}
\noindent
Let $\varpi_{\{y_1,\ldots,y_n\}} := \sum_{i=1}^n \delta_{y_i} \in \Omega.$ Then the values of the deterministic path $\lambda(\varpi_{\{y_1,\ldots,y_n\}})$ (resulting from the evaluation of $\lambda$ at the specific $\omega=\varpi_{\{y_1,\ldots,y_n\}}$) at times $s_1,\ldots,s_n$ is given by 
$$ (\lambda_{s_1}(\varpi_{\{y_1,\ldots,y_n\}}),\ldots,\lambda_{s_n}(\varpi_{\{y_1,\ldots,y_n\}}))=(a_1^{\{y_1,\ldots,y_n\}},\ldots,a_n^{\{y_1,\ldots,y_n\}}).$$
In addition 
\begin{equation}
\label{eq:lambdaspecial}
\lambda_t (\varpi_{\{y_1,\ldots,y_n\}}) = \mu + \sum_{i=1}^{n-1} \Phi(t-s_i) \ind{\theta_i \leq a_i^{\{y_1,\ldots,y_n\}}} \ind{s_i < t}, \quad \forall t\geq s_n.
\end{equation}
\end{prop}

\begin{proof}
Let $t \geq 0$. By definition of $\lambda$,

\begin{align*}
& \lambda_t (\varpi_{\{y_1,\ldots,y_n\}}) \\
&:= \mu + \left(\int_{(0,t)} \Phi(t-u) dH_u\right)(\varpi_{\{y_1,\ldots,y_n\}}) \\
&=\mu + \left(\int_{(0,t)} \Phi(t-u) \ind{\theta \leq \lambda_u} N(du,d\theta)\right)(\varpi_{\{y_1,\ldots,y_n\}}) \\
&=\mu + \int_{(0,t)} \Phi(t-u) \ind{\theta \leq \lambda_u(\varpi_{\{y_1,\ldots,y_n\}})} (N(du,d\theta)(\varpi_{\{y_1,\ldots,y_n\}})) \\
&=\mu + \int_{(0,t)} \Phi(t-u) \ind{\theta \leq \lambda_u(\varpi_{\{y_1,\ldots,y_n\}})} (\varpi_{\{y_1,\ldots,y_n\}})(du,d\theta)\\
&=\mu + \int_{(0,s_1)} \Phi(t-u) \ind{\theta \leq \lambda_u(\varpi_{\{y_1,\ldots,y_n\}})} \ind{u < t} (\varpi_{\{y_1,\ldots,y_n\}})(du,d\theta)\\
&+\sum_{i=1}^{n-1} \int_{[s_i,s_{i+1})} \Phi(t-u) \ind{\theta \leq \lambda_u(\varpi_{\{y_1,\ldots,y_n\}})} \ind{u < t} (\varpi_{\{y_1,\ldots,y_n\}})(du,d\theta) \\
&=\mu + \sum_{i=1}^{n-1} \int_{[s_i,s_{i+1})} \Phi(t-u) \ind{\theta \leq \lambda_u(\varpi_{\{y_1,\ldots,y_n\}})} \ind{u < t} (\varpi_{\{y_1,\ldots,y_n\}})(du,d\theta) \\
&=\mu + \sum_{i=1}^{n-1} \Phi(t-s_i) \ind{\theta_i \leq \lambda_{s_i}(\varpi_{\{y_1,\ldots,y_n\}})} \ind{s_i < t}.
\end{align*}
In addition, by definition, of $\lambda$, for any $i$, $\lambda_{s_i}(\varpi_{\{y_1,\ldots,y_n\}}) = \lambda_{s_i}(\varpi_{(y_1,\ldots,y_{i-1})})$. Hence, the evaluation of $\lambda$ at the specific path $\varpi_{\{y_1,\ldots,y_n\}}$ is the deterministic path completely determined by its value at the dates $s_1,\ldots,s_n$. Indeed, 
$$ \lambda_t (\varpi_{\{y_1,\ldots,y_n\}}) = \mu, \quad \forall t\in [0,s_1],$$
in particular $a_1:=\lambda_{s_1} (\varpi_{\{y_1,\ldots,y_n\}}) = \mu$.
From this we deduce that for $t\in (s_1,s_2]$,
$$ \lambda_t (\varpi_{\{y_1,\ldots,y_n\}}) = \mu + \Phi(t-s_1) \ind{\theta_1 \leq \lambda_{s_1}(\varpi_{\{y_1,\ldots,y_n\}})} = \mu + \Phi(t-s_1) \ind{\theta_1 \leq \mu}.$$
In particular $a_2:=\lambda_{s_2} (\varpi_{\{y_1,\ldots,y_n\}}) = \mu + \Phi(s_2-s_1) \ind{\theta_1 \leq \mu} = \mu + \Phi(s_2-s_1) \ind{\theta_1 \leq a_2}$. By induction we get that for $t\in (s_j,s_{j+1}]$ ($j\in \{1,\cdots,n-1\}$),
$$ \lambda_t (\varpi_{\{y_1,\ldots,y_n\}}) = \mu + \sum_{i=1}^{j} \Phi(t-s_i) \ind{\theta_i \leq a_i} \ind{s_i < t},$$
with $a_i:=\lambda_{s_i}(\varpi_{\{y_1,\ldots,y_n\}})$. In other words, $(a_1,\ldots,a_n)$ solves the triangular system of the statement.
\end{proof}

\begin{theorem}$[$Pseudo-chaotic expansion for linear Hawkes processes$]$
\label{th:explicitHawkes}\\
Let $\Phi$ as in Assumption \ref{assumption:Phi} and $\mu>0$. Assume in addition that\footnote{with classical notations $\|\Phi\|_\infty:=\sup_{t \geq 0} \Phi(t)$} $\|\Phi\|_\infty < +\infty$. Let $(H,\lambda)$ be the unique solution of 
\begin{equation}
\label{eq:Hawkes}
\left\lbrace
\begin{array}{l}
H_t = \int_{(0,t]\times \real} \ind{\theta\leq \lambda_s} N(ds,d\theta),\\\\
\lambda_t = \mu +\int_{(0,t)\times\real_+} \Phi(t-s) dH_s,\quad t\geq 0
\end{array}
\right.
\end{equation}
Then $H$ is a linear Hawkes process with intensity $\lambda$ in the sense of Definition \ref{def:standardHawkes}. For any $T>0$, $H_T$ admits the pseudo-chaotic expansion below : 
\begin{equation}
\label{eq:pseudochaosHawkes}
H_T = \sum_{k=1}^{+\infty} \int_{\mathbb X^k} \frac{1}{k!} c_k(x_1,\ldots,x_k) N(dx_1) \cdots N(dx_k),
\end{equation}
$$
\left\lbrace
\begin{array}{l}
c_1(x_1) =\ind{\theta_{1}\leq \mu},\\\\
c_k(x_1,\ldots,x_k) =\displaystyle{(-1)^{k-1} \ind{\theta_{k}\leq \mu} + \sum_{n=1}^{k-1} \sum_{\{y_1,\ldots,y_n\} \subset \{x_{(1)},\ldots,x_{(k-1)}\}} (-1)^{k-1-n} \ind{\theta_{k}\leq \lambda_{x_k} (\varpi_{\{y_1,\ldots,y_n\}})}}, \quad k\geq 2
\end{array}
\right.
$$
where :
\begin{itemize}
\item[-] We recall Notation \ref{notation:ordered} for $(x_{(1)},\ldots,x_{(k)})$
\item[-] Notation $\sum_{\{y_1,\ldots,y_n\} \subset \{x_1,\ldots,x_{k-1}\}}$ stands for the sum over all subsets $\{y_1,\ldots,y_n\}$ of cardinal $n$ of $\{x_1,\ldots,x_{k-1}\}$ 
\item[-] $\lambda_t (\varpi_{\{y_1,\ldots,y_n\}})$ is given by (\ref{eq:lambdaspecial}) in Proposition \ref{prop:systemHawkesdeterministe}.
\end{itemize}
\end{theorem}

\begin{proof}
First by \cite{Bremaud_Massoulie,Costa_etal,Hillairet_Reveillac_Rosenbaum} the system of SDEs (\ref{eq:Hawkes}) admits a unique solution which is a Hawkes process (we refer to \cite{Hillairet_Reveillac_Rosenbaum} for more details on the construction with pathwise uniqueness). Fix $T>0$. Then for $M > \mu$, we set : 
$$ \Omega^M:=\left\{\sup_{t\in [0,T]} \lambda_t \leq M\right\} \subset \Omega. $$
By Markov's inequality
$$ \P[\Omega \setminus \Omega^M] \leq \frac{\E\left[\sup_{t\in [0,T]} \lambda_t\right]}{M} \leq \frac{\mu+\|\Phi\|_\infty \E\left[H_T\right]}{M} \leq \frac{\mu+\|\Phi\|_\infty \|\Phi\|_1 (1-\|\Phi\|_1)^{-1}}{M}.$$
Letting $\bar{\Omega}:=\lim_{M \to +\infty} [\Omega\setminus \Omega^M]$ where the limit is understood as an decreasing sequence of sets, $\P[\bar{\Omega}]=0$.\\\\
\noindent
Fix $M\geq \mu$ and set $(H^M,\lambda^M)$ the unique solution to 
\begin{equation}
\label{eq:HawkesM}
\left\lbrace
\begin{array}{l}
H_t^M = \int_{(0,t]\times [0,M]} \ind{\theta\leq \lambda_s} N(ds,d\theta),\\\\
\lambda_t^M = \mu +\int_{(0,t)\times\real_+} \Phi(t-s) dH_s^M,\quad t\in [0,T].
\end{array}
\right.
\end{equation}
By construction $H^M$ is a counting process with intensity $\lambda \wedge M$ and $(H^M,\lambda^M) = (H,\lambda)$ on $\Omega^M$ by uniqueness of the solution to the SDE. Hence by Theorem \ref{th:pseudochaoticcounting}, $H_T^M$ admits a pseudo-chaotic expansion with respect to $N^{T,M}$ and 
\begin{equation}
\label{eq:pseudodecompositionckcountingtemp}
H_T^M = \sum_{k=1}^{+\infty} \int_{([0,T]\times[0,M])^k} \frac{1}{k!} c_k(x_1,\ldots,x_k) N(dx_1) \cdots N(dx_k),
\end{equation}
\begin{equation}
\label{eq:ckcountingtemp}
c_k(x_1,\ldots,x_k) := \E\left[L^{T,M} D_{(x_{(1)},\ldots,x_{(k-1)})}^{k-1} \ind{\theta_{(k)}\leq \lambda_{(t_k)}}\right], \quad \forall (x_1,\ldots,x_k)\in ([0,T]\times[0,M])^k,
\end{equation}
with the ordering convention of Notation \ref{notation:ordered}. We should mention that the only dependency on $M$ in the coefficients $c_k$ is simply the domain of the variables $(x_1,\ldots,x_k)$ in $([0,T]\times[0,M])^k$. For such $k$ and $(x_1,\ldots,x_k)$ (where for simplicity we assume that $(x_{(1)},\ldots,x_{(k)}) = (x_{1},\ldots,x_{k})$) we have using Proposition \ref{prop:Dnaltenative} that for $\omega \in \Omega^M$,
\begin{align*}
&L^{T,M}(\omega) (D_{(x_{1},\ldots,x_{k-1})}^{k-1} \ind{\theta_{k}\leq \lambda_{t_k}})(\omega) \\
&=L^{T,M}(\omega)  \sum_{J \subset \{1,\cdots,k-1\}} (-1)^{k-1-|J|} \ind{\theta_{k}\leq \lambda_{t_k}(\omega+ \sum_{j\in J} \delta_{x_j})},
\end{align*}
where the sum is over all subsets $J$ of $\{1,\cdots,k-1\}$ including the empty set which is of cardinal $0$. Hence,
\begin{align*}
&L^{T,M}(\omega) (D_{(x_{1},\ldots,x_{k-1})}^{k-1} \ind{\theta_{k}\leq \lambda_{t_k}})(\omega) \\
&=\exp(MT) \ind{(N([0,T]\times[0,M])(\omega))=0} \sum_{J \subset \{1,\cdots,k-1\}} (-1)^{k-1-|J|} \ind{\theta_{k}\leq \lambda_{t_k}(\omega + \sum_{j\in J} \delta_{x_j})} \\
&=\exp(MT) \ind{\omega([0,T]\times[0,M])=0} \sum_{J \subset \{1,\cdots,k-1\}} (-1)^{k-1-|J|} \ind{\theta_{k}\leq \lambda_{t_k}(\omega + \sum_{j\in J} \delta_{x_j})} \\
&=\exp(MT) \ind{\omega([0,T]\times[0,M])=0} \sum_{J \subset \{1,\cdots,k-1\}} (-1)^{k-1-|J|} \ind{\theta_{k}\leq \lambda_{t_k}(\sum_{j\in J} \delta_{x_j})}.
\end{align*}
Recall that $\E[\exp(MT) \ind{\omega([0,T]\times[0,M])=0}] =1$.
In other words the effect of $L^{T,M}$ is to freeze the evaluation of the intensity process $\lambda$ on a specific outcome given by the atoms $(x_{1},\ldots,x_{k-1})$. Taking the expectation and  reorganizing the sum above we get 
\begin{align*}
c_k(x_1,\ldots,x_k) &= \E\left[L^{T,M} D_{(x_{1},\ldots,x_{k-1})}^{k-1} \ind{\theta_{k}\leq \lambda_{t_k}} \right] \\
&=\sum_{J \subset \{1,\cdots,k-1\}} (-1)^{k-1-|J|} \ind{\theta_{k}\leq \lambda_{t_k}(\sum_{j\in J} \delta_{x_j})}\\
&=(-1)^{k-1} \ind{\theta_{k}\leq \mu}+\sum_{n=1}^{k-1} \sum_{\{y_1,\ldots,y_n\} \subset \{x_1,\ldots,x_{k-1}\}} (-1)^{k-1-n} \ind{\theta_{k}\leq \lambda_{t_k}(\varpi_{\{y_1,\ldots,y_n\}}}\\
&=(-1)^{k-1} \ind{\theta_{k}\leq \mu}+\sum_{n=1}^{k-1} \sum_{\{y_1,\ldots,y_n\} \subset \{x_1,\ldots,x_{k-1}\}} (-1)^{k-1-n} \ind{\theta_{k}\leq a_k^{\{y_1,\ldots,y_n\}}}.
\end{align*}
Note that in each term $\ind{\theta_{k}\leq \lambda_{t_k}(\sum_{j\in J} \delta_{x_j})}$, $\sum_{j\in J} \delta_{x_j}$ is deterministic and $\lambda_{t_k}(\sum_{j\in J} \delta_{x_j})$ is explicitly given by the triangular system in Notation \ref{prop:systemHawkesdeterministe}. For $k=1$, the previous expression just reduces to 
$$c_1(x_1) = \E\left[L^{T,M} \ind{\theta_{1}\leq \lambda_{t_1}}\right] = \ind{\theta_{1}\leq \mu}.$$
Finally, as 
$$ H_T(\omega) = H_T^M(\omega), \quad \textrm{ on } \Omega^M,$$
for any $k\geq 1$ 
\begin{align*}
&\int_{([0,T]\times[0,M])^k} \frac{1}{k!} c_k(x_1,\ldots,x_k) N(dx_1) \cdots N(dx_k) \\
&= \int_{\mathbb X^k} \frac{1}{k!} c_k(x_1,\ldots,x_k) N(dx_1) \cdots N(dx_k), \; \textrm{ on } \Omega^M
\end{align*}
and thus the expansion holds true on $\Omega \setminus \bar{\Omega}$ and $\P[\bar{\Omega}]=0$.
\end{proof}

\begin{remark}
The boundedness assumption on $\Phi$ in Theorem \ref{th:explicitHawkes} is not sharpe and can be replaced with any assumption ensuring that for any $T>0$, $\E[\sup_{t\in[0,T]} \lambda_t] <+\infty$.
\end{remark}

\begin{remark}
To the price of cumbersome notations, the previous result can be extended to non-linear Hawkes processes; that is counting processes $H$ with intensity process of the form 
\begin{equation}
\label{eq:HawkesGeneral}
\left\lbrace
\begin{array}{l}
H_t = \int_{(0,t]\times \real_+} \ind{\theta\leq \lambda_s} N(ds,d\theta),\\\\
\lambda_t = h\left(\mu +\int_{(0,t)\times\real_+} \Phi(t-s) dH_s\right),\quad t\in [0,T].
\end{array}
\right.
\end{equation}
where $h:\real \to \real_+$, $\Phi:\real_+ \to \real$ and $\|h\|_1 \|\Phi\|_1 < 1$. Indeed, when computing the coefficients in the expansion, the Poisson measure $N$ is cancelled and involves evaluations of the intensity process at a specific configurations of the form $\varpi_{\{y_1,\ldots,y_n\}}$. This evaluation can be done by a straightforward extension of Proposition \ref{prop:systemHawkesdeterministe} for a non-linear Hawkes process; in other words for both linear or non-linear process the intensity process $\lambda$ is a deterministic function of the  fixed configuration of the form $\varpi_{\{y_1,\ldots,y_n\}}$.
\end{remark}

\begin{discussion}
\label{discussion:avantageforHawkes}
We would like to comment on the advantage of the pseudo-chaotic expansion compared to the usual one for the value $H_T$ of a linear Hawkes process at any time $T$. Recall the two decompositions 
\begin{align*}
H_T &= \E[H_T] + \sum_{j=1}^{+\infty} \frac{1}{j!} I_j(f_j^{H_T}), \quad \textrm{ with } \quad  f_j^{H_T}(x_1,\cdots,x_j) = \E\left[D_{(x_1,\cdots,x_j)}^j H_T\right] \\
&= \sum_{k=1}^{+\infty} \int_{\mathbb X^k} \frac{1}{k!} c_k(x_1,\ldots,x_k) N(dx_1) \cdots N(dx_k)
\end{align*}
For the chaotic expansion, in order to determine each coefficient $f_j$ one has to compute $f_j^{H_T}(x_1,\cdots,x_j) = \E\left[D_{(x_1,\cdots,x_j)}^j H_T\right]$ which turns out to be quite implicit for a general $\Phi$ kernel. Indeed, already for the first coefficient, using Relation (\ref{eq:DH}) we have
\begin{align*}
f_j^{H_T}(x_1) &= \E\left[D_{x_1} H_T\right] \\
& = \E\left[\ind{\theta_1 \leq \lambda_{t_1}} \right] + \int_{(t_1,T]\times \real_+} \E\left[D_{x_{1}}  \ind{\theta\leq \lambda_s} \right] d \theta ds \\
& = \P\left[\theta_1 \leq \lambda_{t_1}\right] + \int_{(t_1,T]} \E\left[D_{x_{1}} \lambda_s \right] ds.
\end{align*}
The quantity $\int_{(t_1,T]} \E\left[D_{x_{1}} \lambda_s \right] ds$ has been computed in \cite{HHKR} however a closed form expression for $\P\left[\theta_1 \leq \lambda_{t_1}\right]$ for any kernel $\Phi$ satisfying Assumption \ref{assumption:Phi}  is unknown to the authors. \\\\
\noindent
In contradistinction, Theorem \ref{th:explicitHawkes} gives an explicit expression for the coefficients $c_k$. In that sense, the pseudo-chaotic expansion (\ref{eq:pseudochaosHawkes}) is an exact representation and an explicit solution to the Hawkes equation formulation as given in Definition \ref{def:standardHawkes}.
\end{discussion}

\section{The pseudo-chaotic expansion and the Hawkes equation}
\label{section:almostHawkes}

The aim of this section is to investigate further the link between a decomposition of the form (\ref{eq:pseudodecompositionckcountingtemp}) that we named pseudo-chaotic expansion and the characterization of a Hawkes process as in Definition \ref{def:standardHawkes}. 
First, let us emphasize that 
both  the standard chaotic expansion and the pseudo-chaotic expansion characterize a given random variable and not a stochastic process. For instance in Theorem \ref{th:explicitHawkes},  the coefficients $c_k$ for the expansion of $H_T$ depend on the time $T$. 
In this section, 
we consider once again the linear Hawkes process,  which is essentially described as a counting process with a specific stochastic intensity like in (\ref{eq:introlambdaHawkes}), and we adopt a different point of view based on population dynamics and branching representation as in \cite{hawkes1974cluster}  or \cite{boumezoued2016population}. Inspired by this branching representation, we  build in Theorem \ref{th:almostHawkes} below a stochastic process \textit{via} its pseudo-chaotic expansion which is an integer-valued piecewise-constant and non-decreasing process with the specific intensity form of a Hawkes process.  Nevertheless, although this stochastic process satisfies the stochastic self-exciting intensity equation which determines a Hawkes process, it fails to be a counting process as it may exhibit jumps larger than one. 
This leaves open further developments for studying the pseudo-chaotic expansion of processes; we refer to Discussion \ref{discussion:finale}.

\subsection{A pseudo-chaotic expansion and branching representation}
\noindent
Throughout this section we will make use of classical stochastic analysis tools hence we describe elements of $\mathbb X$ as $(t,\theta)$ instead of $x$ for avoiding any confusion. The branching representation viewpoint consists in   counting  the number of individuals in generation $n$, where generation $1$ corresponds of the migrants. 
We therefore define a series of counting processes $X_t^{(n)}$ (where $n$ stands for the generation) as follows
\begin{definitionprop}
\label{defintion:definitinHawkes}
(i) Let for any $t\geq 0$
\begin{equation}
\label{eq:X1}
X_t^{(1)}:=\int_{(0,t]\times \real_+} \ind{\theta_1 \leq \mu} N(dv_1,d\theta_1),
\end{equation}
and for $n \geq 2$,
\begin{equation}
\label{eq:Xn}
X_t^{(n)} := \int_{(0,t]\times \real_+} \int_{(0,v_{n}]\times \real_+} \cdots \int_{(0,v_{2}]\times \real_+} \ind{\theta_1 \leq \mu} \prod_{i=2}^n \ind{\theta_i \leq \Phi(v_i-v_{i-1})} N(dv_1,d\theta_1) \cdots N(dv_n,d\theta_n).
\end{equation}
We set in addition 
\begin{equation}
\label{eq:X}
X_t:=\sum_{n=1}^{+\infty} X_t^{(n)},\end{equation}
where the series converges uniformly (in $t$) on compact sets; that is for any $T>0$,
$$ \lim_{p\to+\infty} \E\left[\sup_{t\in [0,T]} \left|X_t-\sum_{n=1}^{p} X_t^{(n)}\right|\right] =0.$$
(ii) We set the $\mathbb F^X$-predictable process
\begin{equation}
\label{eq:lambda}
\ell_t := \mu + \int_{(0,t)} \Phi(t-r) dX_r =\mu + \sum_{n=1}^{+\infty} \int_{(0,t)} \Phi(t-r) dX_r^{(n)}, \quad t\geq 0
\end{equation}
where $\mathbb F^X:=(\mathcal F^X_t)_{t\geq 0}$, with $\mathcal F^X_t:=\sigma(X_s, \; s\leq t)$.
\end{definitionprop}
\noindent The proof of the convergence of the series \eqref{eq:X} is postponed to Section \ref{section:technicallemmata}.  The resulting process $X$ aims at counting the number of individuals in the population, while the predictable process $\ell$ is the candidate to be the self-exciting intensity of the process $X$. This intensity reads as follows
\begin{align*}
&\ell_t \\
&= \mu + \int_{(0,t)} \Phi(t-r) dX_r \\
&=\mu + \int_{(0,t) \times \real_+} \Phi(t-v_1) \ind{\theta_1 \leq \mu} N(dv_1,d\theta_1) \\
&+ \sum_{n=2}^{+\infty} \int_{(0,t]\times \real_+} \Phi(t-v_n) \int_{(0,v_{n}]\times \real_+} \cdots \int_{(0,v_{2}]\times \real_+} \ind{\theta_1 \leq \mu} \prod_{i=2}^n \ind{\theta_i \leq \Phi(v_i-v_{i-1})} N(dv_1,d\theta_1) \cdots N(dv_n,d\theta_n) \\
&=\mu + \sum_{n=1}^{+\infty} \int_{(0,t]\times \real_+} \Phi(t-v_n) \int_{(0,v_{n}]\times \real_+} \cdots \int_{(0,v_{2}]\times \real_+} \ind{\theta_1 \leq \mu} \prod_{i=2}^n \ind{\theta_i \leq \Phi(v_i-v_{i-1})} N(dv_1,d\theta_1) \cdots N(dv_n,d\theta_n).
\end{align*}
The main result of this section is stated below. Its proof based on Lemmata \ref{eq:CalculX} and \ref{lemma:calcullambda} is postponed to Section \ref{section:proofmain}.

\begin{theorem}
\label{th:almostHawkes}
Let $\mu>0$ and $\Phi$ satisfying Assumption \ref{assumption:Phi}. Recall the stochastic process $X$ and $\ell$ defined (in Definition-Proposition \ref{defintion:definitinHawkes}).  We define the process $M:=(M_t)_{t\geq 0}$ by
\begin{equation}
\label{eq:M}
M_t:= X_t-\int_{(0,t)} \ell_u du, \quad t\geq 0.
\end{equation}
Then $X$ is a $\mathbb{N}$-valued non-decreasing process piecewise constant with predictable intensity process $\ell$, in the sense that process $M$ is a $\mathbb F^N$-martingale (and so a $\mathbb F^H$-martingale as  $M$ is $\mathbb F^H$-adapted and $\mathcal F^H_\cdot \subset \mathcal F^N_\cdot $).
\end{theorem}

\begin{discussion}
\label{discussion:finale}
In other words $X$ would be a Hawkes process if it were a counting process, unless it has the same expectation of a Hawkes process. Indeed
\begin{enumerate}
\item some atoms generates simultaneous jumps : any atom $(t_0,\theta_0)$ of $N$ with $\theta_0 \leq \mu$ will generate a jump for $X^{(1)}$ and for all $X^{(n)}$ who have $t_0$ as an ancestor so $X_{t_0}-X_{{t_0}-}$ may be larger than one. 
\item some atoms are ignored by $X$ : by construction any atom $(t_0,\theta_0)$ of $N$ with $\theta > \|\Phi\|_\infty$ is ignored by the process, whereas the area of decision $\ind{\theta\leq \lambda_t}$ is unbounded in the $\theta$-variable for a linear Hawkes process. 
\end{enumerate}
This example leads to a question. More specifically, the more intricate structure of the coefficients in Theorem \ref{th:almostHawkes} for the pseudo-chaotic expansion of a Hawkes process as a sum and differences of indicator functions suggests a necessary algebraic structure with respect to the time variable on the coefficients of the expansion to guarantee the counting-feature of the process. We leave this issue for future research. 
\end{discussion}

\noindent Before handling  in Section  \ref{section:proofmain} the proof of  Theorem \ref{th:almostHawkes},  we start with some useful technical   lemmata.

\subsection{Technical results and proofs}\label{section:technicallemmata}

\begin{lemma}
\label{lemma:magic}
Let $f$ in $L_1(\real_+;dt)$. For any $n\in \mathbb{N}$ with $n \geq 3$, and for any $0\leq s \leq t$,
\begin{equation}
\label{eq:magic}
\int_s^t \int_s^u \Phi_{n-1}(t-r) f(r) dr du = \int_s^{t} \int_s^{v_n} \int_s^{v_{n-1}} \cdots \int_s^{v_2} \prod_{i=2}^{n} \Phi(v_{i}-v_{i-1}) f({v_1}) dv_1 \cdots dv_{n}.
\end{equation} 
\end{lemma}

\begin{proof}
For $g$ a mapping, let $\mathfrak F(g)$ the Fourier transform of $g$.
Fix $s\geq 0$ and $n\geq 3$. Let 
$$ F(u):=\int_s^u \Phi_{n-1}(u-r) f(r) dr, \; u\geq s$$
so that $F = \Phi_{n-1} \ast \tilde f$, with $\tilde f(v):= f(v) \ind{v\geq s}$. We have that 
$$ \mathfrak F(F) = \mathfrak F(\Phi_{n-1}) \mathfrak F(\tilde f) = \mathfrak F(\Phi)^{n-1} \mathfrak F(\tilde f).$$
In addition by definition of mappings $\Phi_i$ (see Relation (\ref{eq:Phin})), each $\Phi_{i}$ is the $i$th convolution of $\Phi$ with itself; hence $\mathfrak F(\Phi_{n-1}) = (\mathfrak F(\Phi))^{n-1}$.
Let : 
$$ G(u):= \int_s^{u} \Phi(u-v_{n-1}) \int_s^{v_{n-1}} \Phi(v_{n-1}-v_{n-2}) \cdots \int_s^{v_2} \Phi(v_{2}-v_1) f({v_1}) dv_1 \cdots dv_{n-1},$$
we immediately get that $\mathfrak F(G) = (\mathfrak F(\Phi))^{n-2} \mathfrak F(\Phi) \mathfrak F(\tilde f) = (\mathfrak F(\Phi))^{n-1} \mathfrak F(\tilde f) = \mathfrak F(F)$. Using the inverse Fourier transform (on the left) we get that $ F(u) = G(u)$ for a.e. $u$ 
leading to $ \int_s^t F(u) du = \int_s^t G(u) du $ which is Relation (\ref{eq:magic}).
\end{proof}

\noindent Lemma \ref{lemma:magic}  allows one to prove  Proposition \ref{defintion:definitinHawkes}, namely to prove that 
the series\\
$ X_t=\sum_{n=1}^{+\infty} X_t^{(n)}$
 converges uniformly (in $t$) on compact sets; that is for any $T>0$,
$$ \lim_{p\to+\infty} \E\left[\sup_{t\in [0,T]} \left|X_t-\sum_{n=1}^{p} X_t^{(n)}\right|\right] =0.$$

\begin{proof}
Set $T>0$. For $p\geq 2$, let $S_p:=\sum_{n=1}^p X^{(n)}$. As each $X^{(n)}$ process is non-negative as a counting process we have that : 
\begin{align*}
&\E\left[\sup_{t \in[0,T]} \left|X_t-S_p(t)\right|\right] \\
&= \E\left[\sup_{t \in[0,T]} \left|\sum_{n=p+1}^{+\infty} X^{(n)}_t \right|\right] \\
&= \sum_{n=p+1}^{+\infty} \E\left[X^{(n)}_T \right] \\
&= \mu \sum_{n=p+1}^{+\infty} \int_0^T \int_0^{v_{n}} \cdots \int_0^{v_{2}} \prod_{i=2}^n \Phi(v_i-v_{i-1}) dv_1 \cdots dv_n \\
&= \mu \sum_{n=p+1}^{+\infty} \int_0^T \int_0^t \Phi_{n-1}(T-r) dr dt, \quad \textrm{ by Lemma \ref{lemma:magic}}\\
&\leq \mu T \sum_{n=p}^{+\infty} \|\Phi_{n}\|_1 = \mu T \frac{\|\Phi\|_1^p}{1-\|\Phi\|_1} \underset{p\to+\infty}{\longrightarrow} 0, \quad \textrm{ by Proposition \ref{prop:Phin}}.
\end{align*}
\end{proof}

\begin{lemma}
\label{lemma:Qnu}
For $n\geq 3$, $0 \leq s \leq u$ we set 
\begin{align*}
&\hspace{-3em} Q(n,u)\\
&\hspace{-3em} :=\E_{s-}\left[\int_{(0,u]} \Phi(u-r) dX_r^{(n-1)} \right]\\
&\hspace{-3em} =\E_{s-}\left[\int_{(0,u]\times \real_+} \hspace{-3em} \Phi(u-v_{n-1}) \int_{(0,v_{n-1}]\times \real_+} \hspace{-1em} \cdots \int_{(0,v_{2}]\times \real_+} \ind{\theta_1 \leq \mu} \prod_{i=2}^{n-1} \ind{\theta_i \leq \Phi(v_i-v_{i-1})} N(dv_1,d\theta_1) \cdots N(dv_{n-1},d\theta_{n-1}) \right].
\end{align*}
We have
\begin{align}
\label{eq:Qnu}
Q(n,u)&=h_u^{s,(n-1)} \nonumber\\
&+ \sum_{i=1}^{n-2} \int_s^u \cdots \int_s^{v_{n-i+1}^*} \Phi(u-v_{n-1}) \prod_{k=n-i+1}^{n-1} \Phi(v_k-v_{k-1}) h_{v_{n-i}}^{s,(n-i-1)} dv_{n-i} \ldots dv_{n-1} \nonumber\\
&+ \mu \int_s^u \int_s^{v_{n-1}} \cdots \int_s^{v_2} \Phi(u-v_{n-1}) \prod_{k=2}^{n-1} \Phi(v_k-v_{k-1}) dv_1\ldots dv_{n-1},
\end{align}
where $v_{n-i+1}^* := v_{n-i+1}$ for $i\neq 1$ and $v_{n-i+1}^* := u$ for $i=1$. An explicit computation gives that Relation (\ref{eq:Qnu}) is valid for $n=2$ using Convention \ref{convention:sums}.
\end{lemma}

\begin{proof}
Using Fubini's theorem\footnote{here Fubini's theorem is used pathwise  as the integral against $N$ are finite-a.e.} we have  
\begin{align*}
&\hspace{-3em} Q(n,u)\\
&\hspace{-3em} = \int_{(0,s]\times \real_+} \hspace{-3em} \Phi(u-v_{n-1}) \int_{(0,v_{n-1}]\times \real_+} \hspace{-1em} \cdots \int_{(0,v_{2}]\times \real_+} \ind{\theta_1 \leq \mu} \prod_{i=2}^{n-1} \ind{\theta_i \leq \Phi(v_i-v_{i-1})} N(dv_1,d\theta_1) \cdots N(dv_{n-1},d\theta_{n-1}) \\
&\hspace{-3em} + \int_s^u \Phi(u-v_{n-1}) \E_{s-}\left[\int_{(0,v_{n-1}]\times \real_+} \hspace{-1em} \cdots \int_{(0,v_{2}]\times \real_+} \ind{\theta_1 \leq \mu} \prod_{i=2}^{n-1} \ind{\theta_i \leq \Phi(v_i-v_{i-1})} N(dv_1,d\theta_1) \cdots N(dv_{n-1},d\theta_{n-1}) \right] dv_{n-1} \\
&\hspace{-3em}=\int_{(0,s]\times \real_+} \hspace{-3em} \Phi(u-v_{n-1}) \int_{(0,v_{n-1}]\times \real_+} \hspace{-1em} \cdots \int_{(0,v_{2}]\times \real_+} \ind{\theta_1 \leq \mu} \prod_{i=2}^{n-1} \ind{\theta_i \leq \Phi(v_i-v_{i-1})} N(dv_1,d\theta_1) \cdots N(dv_{n-1},d\theta_{n-1}) \\
&\hspace{-3em} + \int_s^u \Phi(u-v_{n-1}) \E_{s-}\left[\int_{(0,v_{n-1}]\times \real_+} \Phi(v_{n-1}-v_{n-2}) \int_{(0,v_{n-2}]\times \real_+} \hspace{-1em} \cdots \int_{(0,v_{2}]\times \real_+} \right.\\
&\left. \hspace{10em} \ind{\theta_1 \leq \mu} \prod_{i=2}^{n-2} \ind{\theta_i \leq \Phi(v_i-v_{i-1})} N(dv_1,d\theta_1) \cdots N(dv_{n-1},d\theta_{n-1}) \right] dv_{n-1} \\
&\hspace{-3em}=\int_{(0,s]\times \real_+} \hspace{-3em} \Phi(u-v_{n-1}) \int_{(0,v_{n-1}]\times \real_+} \hspace{-1em} \cdots \int_{(0,v_{2}]\times \real_+} \ind{\theta_1 \leq \mu} \prod_{i=2}^{n-1} \ind{\theta_i \leq \Phi(v_i-v_{i-1})} N(dv_1,d\theta_1) \cdots N(dv_{n-1},d\theta_{n-1}) \\
&\hspace{-3em} + \int_s^u \Phi(u-v_{n-1}) Q(n-1,v_{n-1}) dv_{n-1} \\
&\hspace{-3em} =h_u^{s,(n-1)} + \int_s^u \Phi(u-v_{n-1}) Q(n-1,v_{n-1}) dv_{n-1}.
\end{align*}
So we have proved that 
$$ Q(n,u) = h_u^{s,(n-1)} + \int_s^u \Phi(u-v_{n-1}) Q(n-1,v_{n-1}) dv_{n-1}. $$
We then deduce by induction that 
\begin{align*}
Q(n,u)&=h_u^{s,(n-1)} \\
&+ \sum_{i=1}^{n-2} \int_s^u \cdots \int_s^{v_{n-i+1}^*} \Phi(u-v_{n-1}) \prod_{k=n-i+1}^{n-1} \Phi(v_k-v_{k-1}) h_{v_{n-i}}^{s,(n-i-1)} dv_{n-i} \ldots dv_{n-1} \\
&+ \mu \int_s^u \int_s^{v_{n-1}} \cdots \int_s^{v_2} \Phi(u-v_{n-1}) \prod_{k=2}^{n-1} \Phi(v_k-v_{k-1}) dv_1\ldots dv_{n-1}.
\end{align*}
Indeed, assuming the previous relation is true for $Q(n,u)$ for a given $n\geq 3$ and for any $u$; we have
\begin{align*}
&Q(n+1,u) \\
&= h_u^{s,(n)} + \int_s^u \Phi(u-v_{n}) Q(n,v_{n}) dv_{n} \\
&= h_u^{s,(n)} + \int_s^u \Phi(u-v_{n}) h_{v_n}^{s,(n-1)} dv_{n} \\
&+ \int_s^u \Phi(u-v_{n}) \left[  \sum_{i=1}^{n-2} \int_s^{v_n} \cdots \int_s^{v_{n-i+1}} \Phi(v_n-v_{n-1}) \prod_{k=n-i+1}^{n-1} \Phi(v_k-v_{k-1}) h_{v_{n-i}}^{s,(n-i-1)} dv_{n-i} \ldots dv_{n-1} \right] dv_{n} \\
&+ \mu \int_s^u \Phi(u-v_{n}) \int_s^{v_{n-1}}  \Phi(v_n-v_{n-1}) \prod_{k=2}^{n-1} \Phi(v_k-v_{k-1}) dv_1\ldots dv_{n-1} dv_{n} \\
&= h_u^{s,(n)} + \int_s^u \Phi(u-v_{n}) h_{v_n}^{s,(n-1)} dv_{n} \\
&+ \sum_{i=1}^{n-2} \int_s^u \Phi(u-v_{n}) \int_s^{v_n} \cdots \int_s^{v_n-i+1} \prod_{k=n-i+1}^{n} \Phi(v_k-v_{k-1}) h_{v_{n-i}}^{s,(n-i-1)} dv_{n-i} \ldots dv_{n-1} dv_{n} \\
&+ \mu \int_s^u \int_s^{v_{n-1}} \cdots \int_s^{v_2} \Phi(u-v_{n}) \prod_{k=2}^{n} \Phi(v_k-v_{k-1}) dv_1\ldots dv_{n-1} dv_{n} \\
&= h_u^{s,(n)} \\
&+ \sum_{i=0}^{n-2} \int_s^u \Phi(u-v_{n}) \int_s^{v_n} \cdots \int_s^{v_{n-i+1}} \Phi(v_n-v_{n-1}) \prod_{k=n-i+1}^{n} \Phi(v_k-v_{k-1}) h_{v_{n-i}}^{s,(n-i-1)} dv_{n-i} \ldots dv_{n-1} dv_{n} \\
&+ \mu \int_s^u \int_s^{v_{n-1}} \cdots \int_s^{v_2} \Phi(u-v_{n}) \prod_{k=2}^{n} \Phi(v_k-v_{k-1}) dv_1\ldots dv_{n-1} dv_{n},
\end{align*}
where for $i=0$ we use Convention \ref{convention:sums}. Hence
\begin{align*}
&Q(n+1,u) \\
&= h_u^{s,(n)} \\
&+ \sum_{j=1}^{(n+1)-2} \int_s^u \Phi(u-v_{n}) \int_s^{v_n} \cdots \int_s^{v_{(n+1)-j+1}^*} \prod_{k=(n+1)-j+1}^{n+1} \Phi(v_k-v_{k-1}) h_{v_{n+1-j}}^{s,((n+1)-j-1)} dv_{(n+1)-j} \ldots dv_{n-1} dv_{n} \\
&+ \mu \int_s^u \int_s^{v_{n-1}} \cdots \int_s^{v_2} \Phi(u-v_{n}) \prod_{k=2}^{(n+1)-1} \Phi(v_k-v_{k-1}) dv_1\ldots dv_{(n+1)-1},
\end{align*}
which gives the result.
\end{proof}

\subsection{Proof of Theorem \ref{th:almostHawkes}}
\label{section:proofmain}

The proof consists in showing that the process $M$ defined by (\ref{eq:M}) is a $\mathbb F^N$-martingale, that is for any $0\leq s\leq t$, 
\begin{equation}
\label{eq:martingale}
\E_{s-}\left[X_t-X_s\right] = \int_s^t \E_{s-}[\ell_r] dr,
\end{equation}
where for simplicity $\E_{s-}[\cdot] := \E[\cdot \vert \mathcal F_{s-}^N]$.
This result is a direct consequence of Lemma \ref{eq:CalculX} and \ref{lemma:calcullambda} below in which we compute both terms in (\ref{eq:martingale}). To this end we introduce the following notation

\begin{notation}
\label{eq:hn}
Let $s \geq 0$, $v\geq s$ and $n\geq 1$, we set 
$$ h_v^s:=\sum_{n=1}^{+\infty} h_v^{s,(n)},$$
with
\begin{align*}
&h_v^{s,(n)}\\
&:= \int_{(0,s)} \Phi(v-v_n) dX_{v_n}^{(n)} \\
&=\int_{(0,s)\times \real_+}  \Phi(v-v_n) \int_{(0,v_{n}]\times \real_+} \cdots \int_{(0,v_{2}]\times \real_+} \ind{\theta_1 \leq \mu} \prod_{i=2}^n \ind{\theta_i \leq \Phi(v_i-v_{i-1})} N(dv_1,d\theta_1) \cdots N(dv_n,d\theta_n).
\end{align*} 
\end{notation}
\noindent Lemma \ref{eq:CalculX} first  compute the left-hand side of  in (\ref{eq:martingale}). 
\begin{lemma}
\label{eq:CalculX}
For any $0 \leq s \leq t$ we have that 
\begin{align}
\E_{s-}\left[\int_{(s,t]} dX_r \right] = \int_s^t (\mu+h_u^s) du + \int_s^t \int_s^{u}  \Psi(u-r) (\mu + h_{r}^{s})  dr du.
\end{align}
\end{lemma}

\begin{proof}
We have :
\begin{equation}
\label{eq:cond1}
\E_{s-}\left[\int_{(s,t]} dX_r^{(1)} \right] = \E_{s-}\left[\int_{(s,t]\times \real_+} \ind{\theta_1 \leq \mu} N(dv_1,d\theta_1)\right]=\mu (t-s).
\end{equation}

Let $n\geq 2$.

\begin{align*}
&\hspace{-5em}\E_{s-}\left[\int_{(s,t]} dX_r^{(n)} \right] \\
&\hspace{-5em}=\E_{s-}\left[\int_{(s,t]\times \real_+} \int_{(0,v_{n}]\times \real_+} \cdots \int_{(0,v_{2}]\times \real_+} \ind{\theta_1 \leq \mu} \prod_{i=2}^n \ind{\theta_i \leq \Phi(v_i-v_{i-1})} N(dv_1,d\theta_1) \cdots N(dv_n,d\theta_n) \right] \\
&\hspace{-5em} =\int_s^t \E_{s-}\left[\int_{(0,v_{n}]\times \real_+} \hspace{-3em} \Phi(v_n-v_{n-1}) \int_{(0,v_{n-1}]\times \real_+} \hspace{-1em} \cdots \int_{(0,v_{2}]\times \real_+} \ind{\theta_1 \leq \mu} \prod_{i=2}^{n-1} \ind{\theta_i \leq \Phi(v_i-v_{i-1})} N(dv_1,d\theta_1) \cdots N(dv_{n-1},d\theta_{n-1}) \right] dv_n\\
&\hspace{-5em} =\int_s^t Q(n,v_n) dv_n.
\end{align*}
Hence, by Lemma \ref{lemma:Qnu},
\begin{align}
\label{eq:cond1}
\E_{s-}\left[\int_{(s,t]} dX_r^{(n)} \right] &=\int_s^t h_{v_n}^{s,(n-1)} dv_n \nonumber\\
&+\sum_{i=1}^{n-2} \int_s^t \int_s^{v_n} \cdots \int_s^{v_{n-i+1}} \prod_{k=n-i+1}^{n} \Phi(v_k-v_{k-1}) h_{v_{n-i}}^{s,(n-i-1)} dv_{n-i} \ldots dv_n \nonumber\\
&+\mu \int_s^t \int_s^{v_n} \cdots \int_s^{v_2} \prod_{k=2}^{n} \Phi(v_k-v_{k-1}) dv_1\ldots dv_n.
\end{align}

Using Lemma \ref{lemma:magic},

\begin{align*}
&\sum_{i=1}^{n-2} \int_s^t \int_s^{v_n} \cdots \int_s^{v_{n-i+1}} \prod_{k=n-i+1}^{n} \Phi(v_k-v_{k-1}) h_{v_{n-i}}^{s,(n-i-1)} dv_{n-i} \ldots dv_n \\
&=\sum_{i=1}^{n-2} \int_s^t \int_s^{u} \Phi_i(u-r) h_{r}^{s,(n-i-1)} dr du,
\end{align*}
and 
$$\mu \int_s^t \int_s^{v_{n}} \cdots \int_s^{v_2} \prod_{k=2}^{n} \Phi(v_k-v_{k-1}) dv_1\ldots dv_n=\mu \int_s^t \int_s^{u} \Phi_{n-1}(u-r) dr du.$$
Plugging back these expressions in (\ref{eq:cond1}) we get 
\begin{align}
\label{eq:cond2}
&\E_{s-}\left[\int_{(s,t]} dX_r^{(n)} \right] \nonumber\\
&=\int_s^t h_{u}^{s,(n-1)} du +\sum_{i=1}^{n-2} \int_s^t \int_s^{u} \Phi_i(u-r) h_{r}^{s,(n-i-1)} dr du +\mu \int_s^t \int_s^{u} \Phi_{n-1}(u-r) dr du.
\end{align}
We now sum the previous quantity over $n \geq 2$. The main term to be treated is the second one that we treat separately below. 
Note also that using Convention \ref{convention:sums} for $n=2$ we get 
\begin{align}
\label{eq:tempdoublesum}
&\sum_{n=2}^{+\infty} \sum_{i=1}^{n-2} \int_s^t \int_s^{u} \Phi_i(u-r) h_{r}^{s,(n-i-1)} dr du \nonumber\\
&=\sum_{n=3}^{+\infty} \sum_{i=1}^{n-2} \int_s^t \int_s^{u} \Phi_i(u-r) h_{r}^{s,(n-i-1)} dr du \nonumber\\
&= \sum_{n=3}^{+\infty} \sum_{j=1}^{n-2} \int_s^t \int_s^{u} \Phi_{n-j-1}(u-r) h_{r}^{s,(j)} dr du \nonumber\\
&= \sum_{j=1}^{+\infty} \int_s^t \int_s^{u}  h_{r}^{s,(j)} \left(\sum_{n=3}^{+\infty}  \ind{j\leq n-2} \Phi_{n-j-1}(u-r) \right)dr du \nonumber\\
&= \int_s^t \int_s^{u}  h_{r}^{s,(1)} \left(\sum_{n=3}^{+\infty} \Phi_{n-2}(u-r) \right)dr du + \sum_{j=2}^{+\infty} \int_s^t \int_s^{u}  h_{r}^{s,(j)} \left(\sum_{n=j+2}^{+\infty} \Phi_{n-j-1}(u-r) \right)dr du \nonumber\\
&= \int_s^t \int_s^{u}  h_{r}^{s,(1)} \left(\sum_{k=1}^{+\infty} \Phi_{k}(u-r) \right)dr du + \sum_{j=2}^{+\infty} \int_s^t \int_s^{u}  h_{r}^{s,(j)} \left(\sum_{k=1}^{+\infty} \Phi_{k}(u-r) \right)dr du \nonumber\\
&= \int_s^t \int_s^{u}  h_{r}^{s,(1)} \Psi(u-r) dr du + \sum_{j=2}^{+\infty} \int_s^t \int_s^{u}  h_{r}^{s,(j)} \Psi(u-r) dr du \nonumber\\
&=\sum_{j=1}^{+\infty} \int_s^t \int_s^{u}  h_{r}^{s,(j)} \Psi(u-r) dr du \nonumber\\
&=\int_s^t \int_s^{u}  \Psi(u-r) h_{r}^{s}  dr du.
\end{align}
With these computations at hand, Relations (\ref{eq:cond1}) and (\ref{eq:cond2}) lead to
\begin{align*}
&\E_{s-}\left[\int_{(s,t]} dX_r \right] \\
&=\E_{s-}\left[\int_{(s,t]} dX_r^{(n)} \right] + \sum_{n=2}^{+\infty} \E_{s-}\left[\int_{(s,t]} dX_r^{(n)} \right] \\
&= \int_s^t \mu du + \sum_{n=2}^{+\infty} \int_s^t h_{u}^{s,(n-1)} du +  \sum_{n=2}^{+\infty} \sum_{i=1}^{n-2} \int_s^t \int_s^{u} \Phi_i(u-r) h_{r}^{s,(n-i-1)} dr du +\mu \sum_{n=2}^{+\infty} \int_s^t \int_s^{u} \Phi_{n-1}(u-r) dr du \\
&= \int_s^t (\mu+h_u^s) du + \int_s^t \int_s^{u}  \Psi(u-r) (\mu + h_{r}^{s})  dr du,
\end{align*}
which concludes the proof.
\end{proof}

\noindent We now compute the right-hand side in (\ref{eq:martingale}).

\begin{lemma}
\label{lemma:calcullambda}
For any $0 \leq s \leq t$ we have that :
\begin{align}
\E_{s-}\left[\int_s^t \ell_r dr \right] = \int_s^t (\mu+h_u^s) du + \int_s^t \int_s^{u}  \Psi(u-r) (\mu + h_{r}^{s})  dr du.
\end{align}
\end{lemma}

\begin{proof}
The proof is rather similar to the one of Lemma \ref{eq:CalculX}, we provide a proof for the sake of completeness.\\
Let $0 \leq s \leq r \leq t$. Recall Notation \ref{eq:hn}. We have 
\begin{align}
\label{eq:calcultemplambda}
\E_{s-}\left[\ell_r\right] &= \E_{s-}\left[\mu + \int_{(0,r)} \Phi(r-u) dX_u \right] \nonumber\\
&= \mu + \sum_{n=1}^{+\infty} \E_{s-}\left[\int_{(0,r)} \Phi(r-u) dX_u^{(n)} \right] \nonumber\\
&= \mu + \sum_{n=1}^{+\infty} \int_{(0,s)} \Phi(r-u) dX_u^{(n)} +\sum_{n=1}^{+\infty} \E_{s-}\left[\int_{(s,r)} \Phi(r-u) dX_u^{(n)} \right] \nonumber\\
&= \mu + h_r^s + \sum_{n=1}^{+\infty} \E_{s-}\left[\int_{(s,r)} \Phi(r-u) dX_u^{(n)} \right].
\end{align}
Let $n\geq 2$,
\begin{align*}
&\E_{s-}\left[\int_{(s,r)} \Phi(r-u) dX_u^{(n)} \right]\\
&= \E_{s-}\left[\int_{(s,r)} \Phi(r-v_n) \int_{(0,v_n]\times \real_+} \hspace{-1em} \cdots \int_{(0,v_{2}]\times \real_+} \ind{\theta_1 \leq \mu} \prod_{i=2}^{n} \ind{\theta_i \leq \Phi(v_i-v_{i-1})} N(dv_1,d\theta_1) \cdots N(dv_{n},d\theta_{n}) \right] \\
&= \int_s^r \Phi(r-v_n) \E_{s-}\left[\int_{(0,v_n]\times \real_+} \Phi(v_n-v_{n-1}) \int_{(0,v_{n-1}]\times \real_+} \cdots \int_{(0,v_{2}]\times \real_+} \right.\\
&\left. \hspace{5em} \ind{\theta_1 \leq \mu} \prod_{i=2}^{n-1} \ind{\theta_i \leq \Phi(v_i-v_{i-1})} N(dv_1,d\theta_1) \cdots N(dv_{n-1},d\theta_{n-1}) \right] dv_n \\
&=\int_s^r \Phi(r-v_n) Q(n,v_n) dv_n\\
&=\int_s^r \Phi(r-v_{n}) h_{v_{n}}^{s,(n-1)} dv_{n}\\
&+\sum_{i=1}^{n-2} \int_s^r \Phi(r-v_{n}) \int_s^{v_{n}} \cdots \int_s^{v_{n-i+1}^*} \prod_{k=n-i+1}^{n} \Phi(v_k-v_{k-1}) h_{v_{n-i}}^{s,(n-i-1)} dv_{n-i} \ldots dv_{n-1}  dv_{n}\\
&+ \mu \int_s^r \Phi(r-v_{n}) \int_s^{v_{n}} \int_s^{v_{n-1}} \cdots \int_s^{v_2} \prod_{k=2}^{n} \Phi(v_k-v_{k-1}) dv_1\ldots dv_{n-1} dv_{n}
\end{align*}
where the last equality follows from Lemma \ref{lemma:Qnu}.
Integrating the previous expression in $r$ on $(s,t]$ and using Lemma \ref{lemma:magic} one gets
\begin{align*}
&\int_s^t \E_{s-}\left[\int_{(s,r)} \Phi(r-u) dX_u^{(n)} \right] dr\\
&=\int_s^t \int_s^{v_{n}} \Phi(v_n-v_{n+1}) h_{v_{n+1}}^{s,(n-1)} dv_{n+1} dv_n\\
&+\sum_{i=1}^{n-2} \int_s^t \int_s^{v_{n+1}} \int_s^{v_{n}} \cdots \int_s^{v_{n-i+1}^*} \prod_{k=n-i+1}^{n+1} \Phi(v_k-v_{k-1}) h_{v_{n-i}}^{s,(n-i-1)} dv_{n-i} \ldots dv_{n-1} dv_{n} dv_{n+1}\\
&+ \mu \int_s^t \int_s^{v_{n+1}} \int_s^{v_{n}} \int_s^{v_{n-1}} \cdots \int_s^{v_2} \prod_{k=2}^{n+1} \Phi(v_k-v_{k-1}) dv_1\ldots dv_{n-1} dv_{n} dv_{n+1}\\
&=\int_s^t \int_s^u \Phi(u-r) h_{r}^{s,(n-1)} dr du\\
&+\sum_{i=1}^{n-2} \int_s^t \int_s^u \Phi_{i+1}(u-r) h_{r}^{s,(n-i-1)} dr du \\
&+ \mu \int_s^t \int_s^u \Phi_n(u-r) dr du
\end{align*}
Using the same computations than (\ref{eq:tempdoublesum}) we deduce (using Convention \ref{convention:sums}) that 
\begin{align*}
&\sum_{n=1}^{+\infty} \int_s^t \E_{s-}\left[\int_{(s,r)} \Phi(r-u) dX_u^{(n)} \right] dr\\
&=\int_s^t \E_{s-}\left[\int_{(s,r)} \Phi(r-u) dX_u^{(1)} \right] + \sum_{n=2}^{+\infty} \int_s^t \E_{s-}\left[\int_{(s,r)} \Phi(r-u) dX_u^{(n)} \right] dr\\
&=\mu \int_s^t \int_s^u \Phi(r-u) dr du+\sum_{n=2}^{+\infty} \int_s^t \int_s^u \Phi(u-r) h_{r}^{s,(n-1)} dr du\\
&+\sum_{n=3}^{+\infty} \sum_{i=1}^{n-2} \int_s^t \int_s^u \Phi_{i+1}(u-r) h_{r}^{s,(n-i-1)} dr du + \sum_{n=2}^{+\infty} \mu \int_s^t \int_s^u \Phi_n(u-r) dr du \\
&=\mu \int_s^t \int_s^u \Psi(u-r) dr du \\
&+\sum_{n=2}^{+\infty} \int_s^t \int_s^u \Phi(u-r) h_{r}^{s,(n-1)} dr du\\
&+\sum_{n=3}^{+\infty} \sum_{j=2}^{(n+1)-2} \int_s^t \int_s^u \Phi_{j}(u-r) h_{r}^{s,(n+1-j-1)} dr du \\
&=\mu \int_s^t \int_s^u \Psi(u-r) dr du +\sum_{n=2}^{+\infty} \int_s^t \int_s^u \Phi(u-r) h_{r}^{s,(n-1)} dr du +\sum_{n=2}^{+\infty} \sum_{j=2}^{n-2} \int_s^t \int_s^u \Phi_{j}(u-r) h_{r}^{s,(n-j-1)} dr du\\
&=\mu \int_s^t \int_s^u \Psi(u-r) dr du +\sum_{n=2}^{+\infty} \sum_{j=1}^{n-2} \int_s^t \int_s^u \Phi_{j}(u-r) h_{r}^{s,(n-j-1)} dr du\\
&=\mu \int_s^t \int_s^u \Psi(u-r) dr du + \int_s^t \int_s^u \Psi(u-r) h_{r}^{s} dr du \\
&=\int_s^t \int_s^u \Psi(u-r) (\mu+h_{r}^{s}) dr du.
\end{align*}
where we have used Relation (\ref{eq:tempdoublesum}). Thus we have proved that 
\begin{align*}
\int_s^t \E_{s-}\left[\int_{(s,r)} \Phi(r-u) dX_u \right] dr  &= \sum_{n=1}^{+\infty} \int_s^t \E_{s-}\left[\int_{(s,r)} \Phi(r-u) dX_u^{(n)} \right] dr \\
&=\int_s^t \int_s^u \Psi(u-r) (\mu+h_{r}^{s}) dr du.
\end{align*}
Hence coming back to Relation (\ref{eq:calcultemplambda}) we obtain
$$\E_{s-}\left[\int_s^t \ell_r dr\right]  = \int_s^t (\mu + h_u^s) du + \int_s^t \int_s^u \Psi(u-r) (\mu+h_{r}^{s}) dr du.$$
\end{proof}


\end{document}